\documentclass[11pt]{amsart}
\usepackage{amsmath, amsthm, mathabx,amscd, amsfonts, amssymb, hyperref, mathrsfs, textcomp, epsfig, a4wide, graphicx, IEEEtrantools, tikz, verbatim, xypic, extpfeil}
\usepackage[mathscr]{euscript}
\newcommand\numberthis{\addtocounter{equation}{1}\tag{\theequation}}

\usetikzlibrary{matrix,decorations.pathreplacing,angles,quotes}

\newtheorem{thm}{Theorem}
\newtheorem{prop}{Proposition}
\newtheorem{lemma}{Lemma}
\newtheorem{quest}{Question}
\newtheorem{cor}{Corollary}
\newcommand{\spin}{\mathfrak{s}}
\theoremstyle{definition}
\newtheorem{defn}{Definition}

\theoremstyle{remark}
\newtheorem{remark}{Remark}
\newtheorem{example}{Example}

     \RequirePackage{rotating}                   
    \def\HSt{%
       \setbox0=\hbox{$\widehat{\mathit{HS}}$}
       \setbox1=\hbox{$\mathit{HS}$}
       \dimen0=1.1\ht0
       \advance\dimen0 by 1.17\ht1
       \smash{\mskip2mu\raise\dimen0\rlap{%
          \begin{turn}{180}
              {$\widehat{\phantom{\mathit{HS}}}$}
           \end{turn}} \mskip-2mu    
                \mathit{HS}
    }{\vphantom{\widehat{\mathit{HS}}}}{}}

     \RequirePackage{rotating}                   
    \def\HMt{%
       \setbox0=\hbox{$\widehat{\mathit{HM}}$}
       \setbox1=\hbox{$\mathit{HM}$}
       \dimen0=1.1\ht0
       \advance\dimen0 by 1.17\ht1
       \smash{\mskip2mu\raise\dimen0\rlap{%
          \begin{turn}{180}
              {$\widehat{\phantom{\mathit{HM}}}$}
           \end{turn}} \mskip-2mu    
                \mathit{HM}
    }{\vphantom{\widehat{\mathit{HM}}}}{}}
    
    \newcommand{\HMb}{\overline{\mathit{HM}}}
\newcommand{\HMf}{\widehat{\mathit{HM}}}
    \newcommand{\HSb}{\overline{\mathit{HS}}}
\newcommand{\HSf}{\widehat{\mathit{HS}}}

\newcommand{\Pin}{\mathrm{Pin}(2)}

\newcommand{\Th}{\Theta^3_H}
\newcommand{\Rint}{\widetilde{\mathcal{R}}}
\newcommand{\Rin}{\mathcal{R}}
\newcommand{\ztwo}{\mathbb{F}}

\newcommand{\V}{\mathcal{V}}
\newcommand{\U}{\mathcal{U}}

\newcommand{\Tor}{\mathrm{Tor}}
\newcommand{\Ainf}{\mathcal{A}_{\infty}}

\newcommand{\x}{\mathbf{x}}
\newcommand{\y}{\mathbf{y}}
\newcommand{\z}{\mathbf{z}}
\newcommand{\w}{\mathbf{w}}

\newcommand{\A}{\mathbf{a}}
\newcommand{\B}{\mathbf{b}}
\newcommand{\m}{\mathfrak{m}}

\begin{document}

\title{Non-formality in $\mathrm{PIN}(2)$-monopole Floer homology} 

\author{Francesco Lin}
\address{Department of Mathematics, Princeton University} 
\email{fl4@math.princeton.edu}

\begin{abstract}
In previous work, we introduced a natural $\Ainf$-structure on the $\Pin$-monopole Floer chain complex of a closed, oriented three-manifold $Y$, and showed that it is non-formal in the simplest case in which $Y$ is the three-sphere $S^3$. In this paper, we provide explicit descriptions of several Massey products induced on homology, and discuss how they can be used to compute the $\Pin$-monopole Floer homology of connected sums in many concrete examples.
\end{abstract}
\maketitle

Starting with Manolescu's disproof of the longstanding Triangulation conjecture \cite{Man2}, the study of $\Pin$-symmetry in Seiberg-Witten theory, where
\begin{equation*}
\Pin=S^1\cup j\cdot S^1\subset \mathbb{H},
\end{equation*}
has spurred a lot of activity, especially in light of its applications to the study of the homology cobordism group $\Th$. The analogous theory of involutive Heegaard Floer homology \cite{HenM} (which heuristically corresponds to a $\mathbb{Z}_4$-equivariant theory, where $\mathbb{Z}_4=\langle j\rangle\subset\Pin$) has also been very successful when addressing such problems. Despite all of this, still very little is known about $\Th$, and among the several natural questions one may ask, the following is particularly interesting.
\begin{quest}\label{quest}
Is there a torsion element in $\Th$ with Rokhlin invariant $1$?
\end{quest}
The negative answer for $2$-torsion elements was provided by Manolescu in \cite{Man2}, and is equivalent to the Triangulation conjecture being false by classic results of Galewski-Stern and Matumoto (see \cite{Man3} for a nice survey). In a related fashion, the interest in Question \ref{quest} stems from the fact that a negative answer would imply the following criterion for triangulability: a closed orientable topological manifold $M$ is triangulable if and only if its Kirby-Siebenmann invariant $\Delta(M)\in H^4(M;\mathbb{Z}/2\mathbb{Z})$ admits a lift to $H^4(M;\mathbb{Z})$. A partial negative answer to the question, when restricting the attention to connected sums of almost rational plumbed three-manifolds, was provided using involutive Heegaard Floer homology in \cite{DaiS}. On the other hand, as the problem involves the Rokhlin invariant, one could expect the full $\Pin$-symmetry, rather than $\mathbb{Z}_4$-symmetry, to play a central role in an approach to its answer.
\\
\par
With Question \ref{quest} as a motivation in mind, we study in this paper the more general problem of understanding the $\Pin$-monopole Floer homology of connected sums. The treatment of such a problem in the analogous setups of $\Pin$-equivariant Seiberg-Witten Floer homology and involutive Heegaard Floer homology can be found in \cite{Sto2}, \cite{HMZ} and \cite{DaiS}. $\Pin$-monopole Floer homology was introduced in \cite{Lin} as a counterpart of Manolescu's invariants in the Morse-theoretic setting of Kronheimer-Mrowka's monopole Floer homology \cite{KM}. In this paper, we are mostly interested in the (completed) invariant $\HSf_{\bullet}(Y,\spin)$ (pronounced \textit{HS-to}) associated to a three-manifold equipped with a self-conjugate spin$^c$ structure $\spin=\bar{\spin}$. This is a graded module over the ring
\begin{equation*}
\Rin=\HSf_{\bullet}(S^3)=\ztwo[[V]][Q]/Q^3
\end{equation*}
where $V$ and $Q$ have degrees respectively $-4$ and $-1$. It was shown in \cite{Lin4} that this package of invariants carries an extremely rich algebraic structure: namely, if we denote by $\hat{C}^{\jmath}_{\bullet}(Y,\spin)$ the chain complex underlying $\HSf_{\bullet}(Y,\spin)$, then $\hat{C}^{\jmath}_{\bullet}(S^3)$ has a natural structure of $\Ainf$-algebra and $\hat{C}^{\jmath}_{\bullet}(Y,\spin)$ is naturally an $\Ainf$-module over it. Furthermore, it was shown in \cite{Lin4} that the $\Ainf$-algebra $\hat{C}^{\jmath}_{\bullet}(S^3)$ is \textit{not} formal (i.e. not quasi-isomorphic to its homology). While the concept of non-formality has a very long history (see for example the celebrated results in \cite{DGMS} and \cite{GMor}), it has recently gained importance in understanding Floer theoretic invariants, especially regarding those arising in symplectic geometry (see for example \cite{AbS} and \cite{LP}).
\\
\par
The goal of the present paper is to explore these non-formality phenomena, and in particular their manifestation at the homology level as Massey products. Our interest in the study of these properties (especially towards Question \ref{quest}) is that Massey products naturally appear when trying to explicitly understand the $\Pin$-monopole Floer homology of connected sums. Indeed, the main result of \cite{Lin4} described the Floer chain complex of a connected sum in terms of the $\Ainf$-tensor product of the Floer complexes of the summands; this naturally leads to a spectral sequence, called the \textit{Eilenberg-Moore spectral sequence}, whose $E^2$-page is
\begin{equation*}
\mathrm{Tor}^{\Rin}_{*,*}(\HSf_{\bullet}(Y_0,\spin_0),\HSf_{\bullet}(Y_1,\spin_1))
\end{equation*}
and converges (up to grading shift) to $\HSf_{\bullet}(Y_0\hash Y_1,\spin_0,\spin_1)$. Non-formality comes into play when studying the successive pages of this spectral sequence: the higher differentials (and the extension problems relating $E^{\infty}$ to the actual group) are naturally described in terms of certain Massey products of the two summands.
\\
\par
While the main results of \cite{Lin4} provides a general, yet not concretely applicable, connected sum formula, the main goal of this paper is to show that in many cases the computations involving the $\Ainf$-structure and the Eilenberg-Moore spectral sequence can be explicitly performed. Towards this end, our exposition will blend general results with concrete examples, and we will discuss how several results proved in the literature with different methods fit in our framework.
\\
\par
Let us discuss the content of the various sections. In Section \ref{review}, we begin by providing a review of the essential aspects of $\Pin$-monopole Floer homology needed in the rest of the paper. Given this background, we show in Section \ref{DescrMassey} that several natural Massey (bi)products (including for example $\langle \x,Q,Q^2\rangle$, when $Q\cdot\x=0$ and $\langle V,\x,Q\rangle$, when $V\cdot\x=Q\cdot\x=0$) can be described in terms of the Gysin exact triangle
\begin{center}
\begin{tikzpicture}
\matrix (m) [matrix of math nodes,row sep=1em,column sep=0.5em,minimum width=2em]
  {
 \HSf_{\bullet}(Y,\spin)  && \HSf_{\bullet}(Y,\spin)\\
  &\HMf_{\bullet}(Y,\spin)\\};
  \path[-stealth]
  (m-1-1) edge node [above]{$\cdot Q$} (m-1-3)
   (m-2-2) edge node [left]{$\pi_*$}(m-1-1)
  (m-1-3) edge node [right]{$\iota$} (m-2-2)  
  ;
\end{tikzpicture}
\end{center}
relating $\HSf_{\bullet}(Y,\spin)$ with the usual monopole Floer homology $\HMf_{\bullet}(Y,\spin)$. As a consequence of this description we obtain an interpretation of the $U$-action on $\HMf_{\bullet}(Y,\spin)$ in terms of the $\Ainf$-structure on $\HSf_{\bullet}(Y,\spin)$ (see for example Proposition \ref{froy}). The Gysin exact triangle can be explicitly understood in several cases including Seifert spaces and spaces obtained by surgery on $L$-space knots, as discussed in Section \ref{examples}. In that section, we also introduce a rather large class of homology spheres called manifolds of simple type $M_n$, and discuss concrete examples (see Proposition \ref{manysimple}). Given this, we turn our attention onto the study of the Eilenberg-Moore spectral sequence. In Section \ref{projres}, we cover the relevant background in homological algebra over our  ring $\Rin$ needed to describe concretely the $E^2$-page of the spectral sequence. This leads up to Section \ref{connsimple}, where we study in detail the higher differentials and extension problems for connected sums with manifold of simple type $M_n$. We will see how the Massey products described in Section \ref{DescrMassey} naturally arise when trying to understand connected sums with this kind of spaces. Finally, we discuss more examples in Section \ref{moreex}, and use them to show how several results from \cite{Sto2} and \cite{DaiS} fit in our framework.
\\
\par
\textit{Acknowledgements. }I would like to thank Matt Stoffregen and Umut Varolgunes for many illuminating conversations. The 2016-2017 special year at the IAS \textit{Homological Mirror Symmetry} was very influential for several aspects of the paper. This work was partially funded by NSF grant DMS-1807242.
\vspace{0.3cm}
\section{A quick review of $\Pin$-monopole Floer homology}\label{review}
In this section, we briefly review the fundamental aspects of $\Pin$-monopole Floer homology which will be needed in the paper, with a particular focus on the results of \cite{Lin4}. We refer the reader to \cite{Lin3} for a more detailed introduction to the subject, and to \cite{Lin} for the details of the construction.
\\
\par
\textit{Formal properties.} To a closed, oriented three-manifold $Y$ equipped with a self-conjugate spin$^c$ structure $\spin$ we associated in \cite{Lin} chain complexes
\begin{equation}\label{flochain}
\check{C}_{\bullet}(Y,\spin),\quad \hat{C}_{\bullet}(Y,\spin),\quad \bar{C}_{\bullet}(Y,\spin)
\end{equation}
equipped with a chain involution $\jmath$. The homology of the chain complexes recovers the monopole Floer homology groups
\begin{equation*}
\HMt_{\bullet}(Y,\spin),\quad \HMf_{\bullet}(Y,\spin),\quad \HMb_{\bullet}(Y,\spin)
\end{equation*}
of \cite{KM}. On the hand, looking at the homology of the invariant chain complexes
\begin{equation*}
\check{C}^{\jmath}_{\bullet}(Y,\spin),\quad \hat{C}^{\jmath}_{\bullet}(Y,\spin),\quad\bar{C}^{\jmath}_{\bullet}(Y,\spin)
\end{equation*}
one obtains the $\Pin$-monopole Floer homology groups fitting in a long exact sequence
\begin{center}
\begin{tikzpicture}
\matrix (m) [matrix of math nodes,row sep=1.5em,column sep=0.5em,minimum width=2em]
  {
 \HSt_{\bullet}(Y,\spin)  && \HSf_{\bullet}(Y,\spin)\\
  &\HSb_{\bullet}(Y,\spin)\\};
  \path[-stealth]
  (m-1-1) edge node [above]{$j_*$} (m-1-3)
   (m-2-2) edge node [left]{$i_*$}(m-1-1)
  (m-1-3) edge node [right]{$p_*$} (m-2-2)  
  ;
\end{tikzpicture}
\end{center}
where the maps $i_*$ and $j_*$ preserve the grading, while $p_*$ has degree $-1$. These are $\mathbb{Q}$-graded modules over $\Rin=\HSf_{\bullet}(S^3)$, where the action is induced in homology by the multiplication map
\begin{equation*}
m_2:\hat{C}^{\jmath}_{\bullet}(Y)\otimes \hat{C}^{\jmath}_{\bullet}(S^3)\rightarrow \hat{C}^{\jmath}_{\bullet}(Y)
\end{equation*}
induced by the cobordism obtained by $([0,1]\times Y)\setminus \mathrm{int}(B^4)$ by attaching cylindrical ends. In \cite{Lin4}, we introduced higher multiplications
\begin{equation*}
m_n:\hat{C}^{\jmath}_{\bullet}(Y)\otimes \hat{C}^{\jmath}_{\bullet}(S^3)^{\otimes n-1}\rightarrow \hat{C}^{\jmath}_{\bullet}(Y)
\end{equation*}
obtained (in the spirit of Baldwin and Bloom's unpublished construction of a monopole category) by looking at an $(n-2)$-dimensional family of metrics and perturbations parametrized by the associahedron $K_n$. It is shown in \cite{Lin4} that in the simplest case in which $Y$ is $S^3$, these operations (which we denote $\mu_n$) define and $\Ainf$-algebra structure on $\hat{C}^{\jmath}_{\bullet}(S^3)$, and for each $Y$ the operations $m_n$ on $\hat{C}^{\jmath}_{\bullet}(Y)$ define a $\Ainf$-module structure over it (see \cite{Lin4} for a quick introduction to $\Ainf$-structures, and \cite{Val} for a more detailed survey). For a fixed choice of data on $S^3$, such an $\Ainf$-module structure on $\hat{C}(Y)$ is well defined up to $\Ainf$-quasi-isomorphism.
\begin{remark}
There are some technical subtleties involved in the construction of \cite{Lin4}, as one needs to impose certain transversality conditions on the chains involved. In particular, the higher composition maps are only partially defined. On the other hand, for the content of this paper (which is mostly algebraic in nature), it will not be harmful to treat the structure constructed in \cite{Lin4} as genuine $\Ainf$-structures.
\end{remark}

\textit{Formality and connected sums}. Recall that an $\Ainf$-algebra $\mathcal{A}$ is called \textit{formal} if it is quasi-isomorphic to its homology (see for example \cite{GMor}). A classical obstruction to formality is provided by Massey products: given the homology classes $[a], [b]$ and $[c]$ in $H_{*}(\mathcal{A})$ such that $[a]\cdot[b]=[b]\cdot[c]=0$, after choosing $r, s$ such that $\partial r=ab$ and $\partial s=bc$, we define the \textit{triple Massey products} to be the homology class
\begin{equation*}
\langle[a],[b],[c]\rangle= [rc+as+\mu_3(a,b,c)].
\end{equation*}
Such a product is well-defined in a suitable quotient of $H_*(\mathcal{A})$. 
Inductively, one can define the $n$-fold Massey products for $n$-uples of homology classes such that all lower Massey products vanish in a consistent way. In the present paper, we will mostly focus on triple and four-fold Massey products. The analogous definitions carry over when defining the Massey products for an $\Ainf$-module $\mathcal{M}$ over $\mathcal{A}$.
\par
It was shown in \cite{Lin4} the $\Ainf$-structure on $ \hat{C}^{\jmath}_{\bullet}(S^3)$ is \textit{not} formal: while the relevant triple Massey products are zero, we have
\begin{equation*}
\langle Q,Q^2,Q,Q^2\rangle=V.
\end{equation*}
Intuitively speaking, this is a cohomological manifestation of the non-triviality of the fiber bundle
\begin{equation*}
\mathbb{R}P^2\hookrightarrow B\Pin\rightarrow \mathbb{H}P^{\infty}.
\end{equation*}
The goal of this paper is to explore the non-formality properties of the $\Ainf$-module $\hat{C}^{\jmath}_{\bullet}(Y)$. This is particularly interesting in light of the the main theorem of \cite{Lin4}, which we now recall.
\begin{thm}\label{mainconn} There exists a quasi-isomorphism
\begin{equation*}
\hat{C}^{\jmath}_{\bullet}(Y_0,\spin_0)\tilde{\otimes}_{\hat{C}(S^3)}\big(\hat{C}^{\jmath}_{\bullet}(Y_1,\spin_1)\big)^{\mathrm{opp}}\cong\hat{C}^{\jmath}_{\bullet}(Y_0\hash Y_1,\spin\hash \spin_1) \langle-1\rangle.
\end{equation*}
where $\mathrm{opp}$ denotes the opposite module.
\end{thm}
Here by $\langle n\rangle$ we denote grading shift \textit{downwards} by $n$, i.e.
\begin{equation*}
(M\langle n\rangle)_d=M_{d-n},
\end{equation*}
while $\tilde{\otimes}$ denotes the $\Ainf$-tensor product. Recall that if $\mathcal{N}$ and $\mathcal{M}$ are (respectively a right and left) $\Ainf$-modules over $\mathcal{A}$, their $\Ainf$-tensor product is defined to be the vector space
\begin{equation*}
\mathcal{N}\tilde{\otimes}\mathcal{M}=\bigoplus_n N\otimes A^n\otimes M
\end{equation*}
equipped with the differential
\begin{align*}
\partial (\x|a_1|\cdots|a_n|\y)&=\sum_{i=0}^{n} m_{i+1}(\x|a_1|\cdots|a_i)|a_{i+1}|\cdots|a_n|\y\\
&+\sum_{i=1}^n\sum_{j=0}^{n-i}\x|a_1|\cdots |a_{i-1}|\mu_{j-i+1}(a_i|\cdots| a_j)|a_{j+1}|\cdots|a_n|\y\\
&+\sum_{i=1}^n\x|a_1|\cdots|a_{i_1}|m_{n-i-1}(a_{i}|\cdots|a_n|\y).
\end{align*}
Here $M,A$ and $N$ denote the underlying $\ztwo$-vector spaces of $\mathcal{M},\mathcal{A}$ and $\mathcal{N}$ and, for simplicity, we will always denote elements of tensor products with bars $\lvert$ instead of $\otimes$s. We can consider the natural filtration given by $F_N=\bigoplus_{n\leq N} N\otimes A^n\otimes M$, and obtain the following.
\begin{cor}
There is a spectral sequence whose $E^2$-page is
\begin{equation*}
\mathrm{Tor}^{\Rin}_{*,*}(\HSf_{\bullet}(Y_0,\spin_0),\HSf_{\bullet}(Y_1,\spin_1))
\end{equation*}
and converges up to grading shift to $\HSf_{\bullet}(Y_0\hash Y_1,\spin_0\hash\spin_1)\langle-1\rangle$.
\end{cor}
We will refer to this as the \textit{Eilenberg-Moore spectral sequence}, see \cite{Lin4} for the heuristic motivation. Here $\mathrm{Tor}^{\Rin}_{*,*}$ is taken in the category of graded $\Rin$-modules, and is therefore a bigraded object. In particular, it can be computed by taking a graded projective resolution of $M$, tensoring it with $N$ and taking the homology of the resulting complex. Recall, as a general fact, given modules $M_0,M_1$ over $\Rin$ the identity
\begin{equation}\label{Tor0}
\mathrm{Tor}^{\Rin}_{0,*}(M_0,M_1)=M_0\otimes_{\Rin} M_1
\end{equation}
holds. The corollary follows from the fact that the $E^1$-page of the spectral sequence associated to the filtration $\{F_N\}$ on $\mathcal{N}\tilde{\otimes}\mathcal{M}$ can be naturally identified with the tensor product of $H_*(\mathcal{M})$ with the bar resolution of $H_*(\mathcal{N})$ (here the key point is that $\hat{C}^{\jmath}_{\bullet}(S^3)$ is cohomologically unital). Of course, $\mathrm{Tor}^{\Rin}_{*,*}$ is independent of the choice of resolution; in Section \ref{projres}, we will discuss some convenient resolutions that allow to compute $\mathrm{Tor}^{\Rin}_{*,*}$ efficiently for our purposes.
\par
While the computation of the $E^2$-page only depends on the module structure, the key observation is that the higher differentials in the Eilenberg-Moore spectral sequence are determined by the Massey products of the two summands. We have, for example, the following (which is proved by a standard staircase argument, see for example \cite{Lin4}).
\begin{lemma}\label{d2}
Suppose we are given $\x\in H_*(\mathcal{M})$, $r_1,\dots, r_n\in H_*(\mathcal{A})$ and $\y\in H_*(\mathcal{N})$ such that
\begin{equation*}
\x r_1=r_1r_2=\dots=r_{n-1}r_n=r_n\y=0.
\end{equation*}
In particular, we have that $\x|r_1|\dots|r_n|\y$ defines a class in $(E^2_{n,*},d_2)$. Then,
\begin{align*}
d_2(\x|r_1|\dots|r_n|\y)&=\langle \x,r_1,r_2\rangle|r_3|\dots|r_n|\y+\x|\langle r_1,r_2,r_3\rangle|r_4|\dots|r_n|\y\\
&+\x|r_1|\dots|\langle r_{n-2},r_{n-1},r_n\rangle|\y+\x|r_1|\dots|\langle r_{n-1},r_n,\y\rangle
\end{align*}
as an element of $E^2_{n-2,*+1}$.
\end{lemma}
In general, not every class in $E^2$ can be described as a simple tensor; on the other hand, one can describe the differential of such classes in terms of certain generalized Massey products, see for example Section \ref{connsimple}.
\par
Similarly, one can also determine the $\Rin$-module structure of the connected sum (i.e. the extension problem for $E^{\infty}$) in terms of the Massey biproducts of the $\Ainf$-bimodule operations
\begin{equation*}
m_{i,j}:\hat{C}^{\jmath}_{\bullet}(S^3)^{\otimes i-1}\otimes \hat{C}^{\jmath}_{\bullet}(Y)\otimes \hat{C}^{\jmath}_{\bullet}(S^3)^{\otimes j-1}\rightarrow \hat{C}^{\jmath}_{\bullet}(Y),
\end{equation*}
see \cite{Lin4} for more details. In fact, Theorem \ref{mainconn} provides an isomorphism of $\Ainf$-bimodules.
\\
\par
\textit{Manolescu correction terms. }From the $\Rin$-module structure, taking as inspiration Fr\o yshov's invariant \cite{Fro},\cite{KM}, one can extract plenty of information regarding cobordisms between manifolds. For simplicitly, let $(Y,\spin)$ be a rational homology sphere $Y$ equipped with a self-conjugate spin$^c$ structure $\spin$ or, equivalently, a spin structure (as $b_1=0$). We can fix an identification, up to grading shift, of graded $\Rin$-modules
\begin{equation*}
\HSb_{\bullet}(Y,\spin)\equiv \Rint
\end{equation*}
where we set
\begin{equation*}
\Rint=\ztwo[V^{-1},V]][Q]/(Q^3),
\end{equation*}
where $\ztwo[V^{-1},V]]$ (which we denote by $\V$) denotes Laurent power series. We have in particular the direct sum of $\ztwo[[V]]$-modules
\begin{equation*}
\Rint=\V\oplus Q\cdot \V\oplus Q^2\cdot \V.
\end{equation*}
Recall that $\HSt_{\bullet}(Y,\spin)$ and $\HSf_{\bullet}(Y,\spin)$ vanish in degrees respectively low and high enough, so that in particular
\begin{align*}
i_*&:\HSb_{\bullet}(Y,\spin)\equiv \Rint\rightarrow \HSt_{\bullet}(Y,\spin)\\
p_*&:\HSf_{\bullet}(Y,\spin)\equiv \Rint\rightarrow \HSb_{\bullet}(Y,\spin).
\end{align*}
are isomorphism is degrees respectively high and low enough.
\begin{defn}
Given a nonzero $r\in\Rint$, we say that $\x\in\HSf_{\bullet}(Y,\spin)$ is \textit{based of type $r$} if $p_*(\x)=r$. If $p_*(\x)=0$, we say that $\x$ is \textit{unbased}. We will say that $\x$ is $\V$, $Q\cdot \V$ or $Q^2\cdot \V$-based according to the where $p_*(\x)$ belongs to.
\end{defn}
We also call the images of $\V$, $Q\cdot \V$ and $Q^2\cdot \V$ in $\HSf_{\bullet}(Y,\spin)$ under $i_*$ respectively the $\alpha$, $\beta$ and $\gamma$-tower. The Manolescu correction terms (defined first in the setting on $\Pin$-equivariant Seiberg-Witten Floer homology \cite{Man2}) are the homology cobordism invariants defined as
\begin{align*}
\alpha&=\frac{1}{2}\mathrm{min} \{\mathrm{deg}(\x)| \x\in \alpha\text{-tower}\}\\
\beta&=\frac{1}{2}(\mathrm{min} \{\mathrm{deg}(\x)| \x\in \beta\text{-tower}\}-1)\\
\gamma&=\frac{1}{2}(\mathrm{min} \{\mathrm{deg}(\x)| \x\in \gamma\text{-tower}\}-2).
\end{align*}
Using the long exact sequence relating the three Floer groups, these numerical invariants can also be described in terms of based elements of $\HSf_{\bullet}(Y,\spin)$ as follows
\begin{align*}
\alpha&=-\frac{1}{2}\left(\mathrm{max} \{\mathrm{deg}(\x)|\text{there exist a $Q^2\cdot \V$-based element $\x$}\}+4\right)\\
\beta&=-\frac{1}{2}\left(\mathrm{max} \{\mathrm{deg}(\x)|\text{there exist a $Q\cdot \V$-based element $\x$}\}+3\right)\\
\gamma&=-\frac{1}{2}\left(\mathrm{max} \{\mathrm{deg}(\x)|\text{there exist a $ \V$-based element $\x$}\}+2\right).
\end{align*}
From this and Poincar\'e duality
\begin{equation*}
\HSt^k(-Y,\spin)\cong\HSf_{-1-k}(Y,\spin)
\end{equation*}
the identities for the correction terms
\begin{align*}
\alpha(-Y,\spin)&=-\gamma(Y,\spin)\\
\beta(-Y,\spin)&=-\beta(Y,\spin)\\
\gamma(-Y,\spin)&=-\alpha(Y,\spin)
\end{align*}
These invariants are lifts of $-\mu(Y,\spin)$, where $\mu$ denotes the Rokhlin invariant, and they provide obstructions to the existence of spin cobordism with $b_2^{+}=0,1,2$, see \cite{Lin},\cite{Lin2}. In particular, they are invariant under homology cobordism.
\par
Of course, one can define analogues of the correction terms also in cases in which $b_1>0$, depending on the structure of $\HSb_{\bullet}(Y,\spin)$. In particular, one can define a correction term for each $\ztwo[V]$ summand of $\HSb_{\bullet}(Y,\spin)$. It is shown in \cite{Lin5} that this only depends on the triple cup product of $Y$, together with the Rokhlin invariants of the $2^{b_1(Y)}$ spin structures inducing $\spin$. For example, there are two cases when $b_1(Y)=1$: when the two spin structures have the same Rokhlin invariant, one obtains six correction terms, while in the case the two spin structures have different Rokhlin one obtains four correction terms (see \cite{Lin2} and also Section \ref{examples}).
\\
\par
\textit{Definition of the Floer chain complexes. }Let us now review the main feature of the Floer chain complexes introduced in \cite{Lin}. The key input of $\Pin$-symmetry is a natural involution on the moduli space of configurations
\begin{equation*}
\jmath: \mathcal{B}(Y,\spin)\rightarrow\mathcal{B}(Y,\spin),
\end{equation*}
whose fixed points are the reducible configurations $[B,0]$ where $B$ is the spin connection of one of the $2^{b_1(Y)}$ spin structures inducing $\spin$. One would like to perform the construction of the Floer chain complexes from \cite{KM} in a way that such a symmetry is preserved. The main complication is that one needs to work with Morse-Bott singularities. In particular, for a generic perturbation, the critical set in the blow up $\mathcal{B}^{\sigma}(Y,\spin)$ consists of:
\begin{itemize}
\item a finite number of pairs of irreducible solutions, acted freely by the action of $\jmath$;
\item for each non spin reducible critical point, an infinite tower of critical points as in the Morse setting. The free action of $\jmath$ of non-spin reducible critical points lifts to an action of the towers;
\item for each spin reducible critical point, an infinite tower of reducible $S^2$s. The involution $\jmath$ acts as the antipodal map on each critical submanifold.
\end{itemize}
The chain complexes with involution (\ref{flochain}) arise as some version of Morse-Bott chain complexes. The underlying vector spaces are generated over $\ztwo$ by suitable geometric chains with values in the critical submanifolds, i.e. smooth maps
\begin{equation*}
f:\sigma\rightarrow C
\end{equation*}
where $\sigma$ is chosen among a suitable generalization of manifolds with boundary and $C$ is a critical submanifold. The differential of such a chain $\sigma$ combines the singular boundary within $C$ together with fibered products with moduli spaces of flows $M(C,C')$ from $C$ to another critical submanifold $C'$, which we consider as singular chains with values in $C'$. In our case, we are naturally lead to deal with $\delta$-chains and the key modification (inspired from \cite{Lip}) is that we consider chains which are non-degenerate, namely both $f(\sigma)$ and $f(\partial\sigma)$ are not contained in the image of smaller dimensional chains.
For our purposes, we will only need that $3$-cycles in our critical $S^2$s are zero at the chain level.
\begin{example}\label{chainvan}
Consider the classes $Q,Q^2\in\Rin=\HSf_{\bullet}(S^3)$. These are represented respectively by a generator in the first and zero dimensional homology of $C_{-1}$, the first unstable critical submanifold (where we consider the round metric on $S^3$, and a small perturbation). Of course, we know $Q\cdot Q^2=0$. In fact, such a product is zero at the chain level: for dimensional reasons, it is a $3$-chain in $C_{-2}$, and because it is closed, it vanishes.
For a similar reason, the triple Massey product $\langle Q,Q^2,Q\rangle$ also vanishes at the chain level.
\end{example}

\vspace{0.3cm}
\section{Description of Massey products}\label{DescrMassey}
In general, the determination of the Massey products of an $\Ainf$-module over an $\Ainf$-algebra is a rather involved process, as it requires the understanding of higher compositions. Our goal in the present section is to show that in the case of $\Pin$-monopole Floer homology, many natural Massey (bi)products can be described very explicitly in terms of the relation with the $U$-action in usual monopole Floer homology. While we will work in the setting of $\HSf_{\bullet}$, all results carry over for $\HSt_{\bullet}$ and $\HSb_{\bullet}$. Before stating the main results of the section, let us recall the Gysin exact sequence
\begin{center}
\begin{tikzpicture}
\matrix (m) [matrix of math nodes,row sep=1em,column sep=0.5em,minimum width=2em]
  {
 \HSf_{\bullet}(Y,\spin)  && \HSf_{\bullet}(Y,\spin)\\
  &\HMf_{\bullet}(Y,\spin)\\};
  \path[-stealth]
  (m-1-1) edge node [above]{$\cdot Q$} (m-1-3)
   (m-2-2) edge node [left]{$\pi_*$}(m-1-1)
  (m-1-3) edge node [right]{$\iota$} (m-2-2)  
  ;
\end{tikzpicture}
\end{center}
introduced in \cite{Lin}. Here the maps $\iota$ and $\pi$ preserve the degree, while multiplication by $Q$ has degree $-1$. It is an exact triangle of $\Rin$-modules where on $\HMf_{\bullet}(Y,\spin)$ we have that $Q$ acts as $0$ and $V$ acts as $U^2$. In the case of a homology sphere of Rokhlin invariant $0$, in degrees between $-4k$ and $-4k-3$ with $k>>0$ the sequence looks like
\begin{center}
\begin{tikzpicture}
\matrix (m) [matrix of math nodes,row sep=1.5em,column sep=3em,minimum width=2em]
  {\cdot& \cdot &\cdot\\
 \ztwo  &\ztwo& \ztwo\\
  \ztwo&\cdot&\ztwo\\
  \ztwo &\ztwo&\ztwo\\};
  \path[-stealth]
  (m-2-1) edge node {} (m-2-2)
  (m-4-2) edge node {} (m-4-3)
  (m-2-3) edge node {} (m-3-1)
  (m-3-3) edge node {} (m-4-1)
  ;
\end{tikzpicture}
\end{center}
where the side columns represents $\HSf_{\bullet}$ while the middle column represents $\HMf_{\bullet}$. Let us record the following general result.

\begin{lemma}\label{Q2}
If $\x\in \HSf_{\bullet}(Y,\spin)$, then $Q^2\cdot\x=\pi_*(U\cdot \iota_*(\x))$.
\end{lemma}

We will prove this later. Granted this, we can will by defining the following Massey operations:
\begin{itemize}
\item if $Q\cdot\x=V\cdot\x=0$, $\langle Q,\x,V\rangle$, which is well defined up to $\mathrm{Im}Q+\mathrm{Im}V$;
\item if $Q\cdot\x=0$, $\langle\x,Q,Q^2\rangle$, well defined up to $\mathrm{Im}Q^2$;
\item if $Q^2\cdot\x=0$, $\langle\x,Q^2,Q\rangle$, well defined up to $\mathrm{Im}Q$.
\item if $Q\cdot\x=0$ and $\langle\x,Q,Q^2\rangle=0$, $\langle\x,Q,Q^2,Q\rangle$ well defined up to $\mathrm{Im}Q$ (recall that $\langle Q,Q^2,Q\rangle=0$, see Example \ref{chainvan}). 
\end{itemize}

On the other hand, using the Gysin exact sequence, we can define the following four operations:
\begin{enumerate}
\item Suppose $Q\cdot\x=V\cdot\x=0$. As $Q\cdot\x=0$, $\x=\pi_*(\y)$ for some $\y$. Then
\begin{equation*}
\pi_*(V\cdot\y)=V\cdot\pi_*(\y)=V\cdot\x=0,
\end{equation*}
so that there exists $\z$ such that $\iota_*(\z)=\y$. We define $\Phi_1(\x)=\z$. It is readily checked that such an element is well defined up to elements in $\mathrm{Im}V+\mathrm{Im}Q$.
\item Suppose $Q\cdot\x$=0.  Then again $\x=\pi_*(\y)$ for some $\y$. We then define $\Phi_2(\x)=\pi_*(U\cdot \y)$. This is well defined up to $\mathrm{Im}Q^2$ in light of Lemma \ref{Q2}.
\item Suppose $Q^2\cdot\x=0$. Then by Lemma \ref{Q2} $\pi_*(U\cdot \iota_*(\x))$, hence $U\cdot \iota_*(\x)=\iota_*(\y)$ for some $\y$. Then we set $\Phi_3(\x)=\y$. This is clearly defined up to $\mathrm{Im}Q$.
\item Suppose $\langle\x,Q,Q^2\rangle=0$. Then, for the map $\Psi_2$ defined in bullet $(2)$ we have that (up to choosing a different $\y$) $\Psi_2(\x)=\pi(U\cdot \y)=0$. Hence $U\cdot\y=\iota(\w)$, and we define $\Psi_4(\x)=\w$.
\end{enumerate}

We will show in Section \ref{examples} that these four operations $\Phi_i$ are explicitly computable in many cases. Their importance for our purposes is the following result.
\begin{thm}\label{Massey}
Let $\x$ be a class in $\HSf_{\bullet}(Y,\spin)$. We have the following identities:
\begin{itemize}
\item if $Q\cdot\x=V\cdot\x=0$, $\langle Q,\x,V\rangle=\Phi_1(\x)$;
\item if $Q\cdot\x=0$, $\langle \x,Q,Q^2\rangle=\Phi_2(\x)$;
\item if $Q^2\cdot\x=0$, $\langle \x,Q^2,Q\rangle=\Phi_3(\x)$;
\item if $\langle \x,Q,Q^2\rangle=0$, $\langle \x,Q,Q^2,Q\rangle=\Phi_4(\x)$.
\end{itemize}
\end{thm}
In fact, while for simplicity we have limited our exposition to Massey products involving only $Q,Q^2$ and $V$, the result naturally generalizes to the analogous Massey products involving $QV^i, Q^2V^j$ and $V^{k+1}$. Let us for example point out how to compute $\langle Q,\x,V^{k+1}\rangle$, where of course we assume $Q\cdot\x=V^{k+1}\cdot\x=0$. As $Q\x=0$ implies that $\x=\pi(\y)$, and because $V^{k+1}\cdot\x=0$, we have that $\pi_*(U^{2k+2}\cdot\y)=0$, so that $U^{2k+2}\cdot\y=\iota (\z)$. We have then $\langle Q,\x,V^{k+1}\rangle=\z$.

\begin{remark}
Looking at the Gysin sequence of $S^3$, we obtain a direct proof (i.e. without relying on an argument involving the Eilenberg-Moore spectral sequence as in \cite{Lin4}) of the fact that $\langle Q,Q^2,Q,Q^2\rangle=V$.
\end{remark}

\begin{remark}\label{kad}
While our main result involves very specific Massey products, one can in general the natural $\Ainf$-structure in homology provided by Kadeishvili's homotopy transfer theorem \cite{Kad} to obtain more information. Let us for example consider the (classical) Massey product $\langle \x|Q^2|Q|Q^2\rangle=0$, where we assume $\langle \x|Q^2|Q\rangle=0$. Then, recalling the vanishing of the triple products in $\Rin$ and the relation $\langle Q|Q^2,Q|Q^2\rangle=V$, we obtain after substituting the latter in the $\Ainf$-relations, the relation
\begin{equation*}
\x\cdot V=\langle \x|Q^2|Q|Q^2\rangle\cdot Q.
\end{equation*}
In several cases, this is enough to determine $\langle \x|Q^2|Q|Q^2\rangle$.
\end{remark}

\vspace{0.3cm}

The proof of this result occupies the rest of the section. Recall first from \cite{Lin} that the Gysin exact sequence arises as the long exact sequence in homology associated to the short exact sequence of chain complexes
\begin{equation}\label{gysinchain}
0\rightarrow \hat{C}^{\jmath}_{\bullet}(Y,\spin)\hookrightarrow\hat{C}_{\bullet}(Y,\spin)\stackrel{1+\jmath}{\longrightarrow}(1+\jmath)\hat{C}_{\bullet}(Y,\spin)\rightarrow0
\end{equation}
where $\hat{C}_{\bullet}(Y,\spin)$ is the Floer chain complex whose homology is $\HMf_{\bullet}(Y,\spin)$ and $\hat{C}^{\jmath}_{\bullet}(Y,\spin)$ and $(1+\jmath)\hat{C}_{\bullet}(Y,\spin)$ both have homology $\HSf_{\bullet}(Y,\spin)$. We first review how the connecting map in the induced long exact sequence is identified with multiplication by $Q$.
Consider a representative $x$ of a class $\x\in\HSf_{\bullet}(Y,\spin)$. Consider its image under the map induced by the cobordism $I\times Y\setminus \mathrm{int}B^4$ with cylindrical ends attached, where we look at the solutions converging to the first negative critical submanifold $C_{-1}$ on the additional incoming $S^3$ end, or, equivalently the element $\hat{m}_2(x\lvert C_{-1})$ obtained from the product map
\begin{equation*}
\hat{m}_2: \hat{C}_{\bullet}(Y,\spin)\otimes \hat{C}_{\bullet}(S^3)\rightarrow \hat{C}_{\bullet}(Y,\spin).
\end{equation*}
by considering the chain $C_{-1}$ on the second factor. As this maps induces the identity in homology, this element is also a representative of $\x$. Recall that $C_{-1}$ is a copy of $S^2$ on which $\jmath$ acts as the antipodal map. Denote by $D^2$ the upper hemisphere, and  by $S^1=\partial D^2$ the equator (notice that the latter is $\jmath$-invariant). We have then $S^2=D^2\cup \jmath(D^2)=(1+\jmath)D^2$, so that
\begin{equation*}
\hat{m}_2(x| (1+\jmath)D^2)=(1+\jmath)\hat{m}_2(x| D^2)=(1+\jmath)(y)
\end{equation*}
where $y=\hat{m}_2(x| D^2)\in \hat{C}_{\bullet}(Y,\spin)$. Now, as $\partial x=0$, we have
\begin{equation*}
\partial y=\partial\left(\hat{m}_2(x| D^2)\right)=\hat{m}_2(x| \partial D^2)=\hat{m}_2(x\times S^1)
\end{equation*}
which is a $\jmath$-invariant cycle, hence in the image of the inclusion $\hat{C}^{\jmath}_{\bullet}(Y,\spin)\hookrightarrow\hat{C}_{\bullet}(Y,\spin)$. By definition, its class in $\HSf_{\bullet}$ represent the image of $\x$ under the boundary map in the induced long exact sequence. On the other hand, as $S^1$ is a representative of $Q$ in $\HSf_{\bullet}(S^3)=\Rin$, $y$ also represents $Q\cdot\x$.
\\
\par
In a similar spirit, we now provide the proof of Lemma \ref{Q2}.
\begin{proof}[Proof of Lemma \ref{Q2}.]
Let $p$ be a base point in $C_{-1}$. Then the $p$ is a cycle representing $U\in\HMf_{\bullet}(S^3)$, while $p\cup \jmath p$ is an invariant cycle representing $Q^2\in\HSf_{\bullet}(S^3)$. If $y\in\hat{C}_{\bullet}(Y,\spin)$ represents $\y\in \HMf_{\bullet}(Y,\spin)$, then $U\cdot \y$ is represented by $\hat{m}_2(y| p)$, while if $x\in \hat{C}^{\jmath}_{\bullet}(Y,\spin)$ represents $\x\in \HSf(Y,\spin)$, then $Q^2\cdot \y$ is represented by $\hat{m}_2(x| p\cup \jmath p)$. On the other hand we have for $x\in \hat{C}^{\jmath}_{\bullet}(Y,\spin)$ that
\begin{equation*}
\hat{m}_2(x| p\cup \jmath p)=\hat{m}_2(x| (1+\jmath)p)=(1+\jmath)\hat{m}_2(x| p),
\end{equation*}
and the result follows.
\end{proof}

With these simple computations in mind, we are ready to prove Theorem \ref{Massey}.

\begin{proof}[Proof of Theorem \ref{Massey}.] Fix a representative $x\in \hat{C}^{\jmath}_{\bullet}(Y,\spin)$ of $\x$.
\par
We begin with proving the first identity. Recall that both multiplication by $V$ on $\HSt_{\bullet}$ and by $U^2$ on $\HMt_{\bullet}$ are both obtained by multiplication by the second negative critical submanifold $C_{-2}$ on the additional incoming $S^3$ end. Let $y=\hat{m}_2(x| D^2)$ as above, so that its image under $1+\jmath$ is a representative of $\x$, and $\partial y=\hat{m}_2(x| S^1)$ represents $Q\cdot\x=0$. Hence is $\hat{m}_2(x| S^1)=\partial t$ for $t\in\hat{C}^{\jmath}(Y,\spin)$. Consider then $\hat{m}_2(x| C_{-2})$, which represents $V\cdot\x=0$ and hence is $\partial s$ for $s\in (1+\jmath)\hat{C}^{\jmath}_{\bullet}(Y,\spin)$. Then by definition the triple Massey product $\langle Q,\x,V\rangle$ is represented by
\begin{equation*}
\hat{m}_2(t| C_{-2})+\hat{m}_2(S^1| s)+\hat{m}_{1,1}(S^1| x| C_{-2})\in \hat{C}^{\jmath}_{\bullet}(Y,\spin).
\end{equation*}
Consider the image of this cycle in $\hat{C}_{\bullet}(Y,\spin)$. We can add to it the boundaries
\begin{align*}
\partial\hat{m}_{1,1}(D^2| x| C_{-2})&=\hat{m}_{1,1}(S^1| x| C_{-2})+\hat{m}_2(\hat{m}_2(D^2| x)| C_{-2})+\hat{m}_2(D^2| \hat{m}_2(x| C_{-2}))\\
\partial\hat{m}_2(D^2| s)&= \hat{m}_2(S^1| s)+\hat{m}_2(D^2| \hat{m}_2(x| C_{-2}))
\end{align*}
so we see that $\iota (\langle Q,\x,V\rangle)$ is represented by 
\begin{equation*}
\hat{m}_2\left(t+\hat{m}_2(D^2| x)| C_{-2}\right)\in\hat{C}_{\bullet}(Y,\spin).
\end{equation*}
On the other hand, by definition $t+\check{m}(D^2\times x)$ is a cycle in $\hat{C}_{\bullet}(Y,\spin)$ whose image under $\pi_*$ is a representative of $\x$; furthermore, the chain above represents its image under the action of $V=U^2$, so that the result follows.
\par
Regarding the second bullet, recall from Example \ref{chainvan} that $Q^2\cdot Q$ is zero at the chain level. Consider as above the cycle $\hat{m}_2(x| S^2)$, and let $t$ a chain such that $\partial t=  \hat{m}_2(x| S^1)$, where again we use $Q\cdot \x=0$. Then the Massey product $\langle\x,Q,Q^2\rangle$ by definition represented by
\begin{equation*}
\hat{m}_2(t| p\cup\jmath p)+\hat{m}_3(t| S^1| p\cup\jmath p)=(1+\jmath)(\hat{m}_2(t| p)+\hat{m}_3(x| S^1| p))\in (1+\jmath)\hat{C}^{\jmath}_{\bullet}(Y,\spin).
\end{equation*}
Of the two natural disks $D^2$ and $\jmath D^2$ whose boundary is $S^1$, we can assume without loss of generality that $\hat{m}_2(D^2| p)=0$. Adding then to the expression above
\begin{equation*}
\partial[(1+\jmath)(m_3(x|D^2|p))]=(1+\jmath)(\hat{m}_3(x|S^1|p)+\hat{m}_2(\hat{m}_2(x| D^2)| p))
\end{equation*}
we see that the Massey product is represented by
\begin{equation*}
(1+\jmath)\check{m}(t+x\times D^2| p).
\end{equation*}
Now $t+x\times D^2\in \hat{C}^{\jmath}_{\bullet}(Y,\spin)$ is again a cycle mapping to a representative of $\x$, and the result follows.
\par
Regarding the third bullet, $Q^2\cdot \x=0$ so by Lemma \ref{Q2} we can consider $z\in (1+\jmath)\hat{C}^{\jmath}_{\bullet}(Y,\spin)$ such that $\partial z=\hat{m}_2(x| p\cup\jmath p)$. Then by definition the Massey product is represented by
\begin{equation*}
\hat{m}_2(z| S^1)+\hat{m}_3(x| p\cup \jmath p| S^1)\in\hat{C}^{\jmath}_{\bullet}(Y,\spin).
\end{equation*}
Consider its image in $\hat{C}_{\bullet}(Y,\spin)$. Adding
\begin{equation*}
\partial \hat{m}_3(x| p\cup \jmath p| D^2)=\hat{m}_3(x| p\cup \jmath p| S^1)+\hat{m}_2(\hat{m}_2(x|  p\cup \jmath p)| D^2)+\hat{m}_2(x\times p),
\end{equation*}
where again we assume $\hat{m}_2(D^2| p)=0$, we see that the image in $\hat{C}_{\bullet}(Y,\spin)$ of the triple Massey product is also represented by
\begin{equation*}
\hat{m}_2(z| S^1)+\hat{m}_2(\hat{m}_2(x|  p\cup \jmath p)| D^2)+\hat{m}_2(x| p)=\hat{m}_2(x| p)+\partial\hat{m}_2(z| D^2),
\end{equation*}
and the result follows.
\par
Finally, suppose $\langle \x, Q,Q^2\rangle=0$. Then, in the notation of the proof of the second bullet, we have that
\begin{equation}\label{fp}
(1+\jmath)[\hat{m}_2(t+\hat{m}_2(x| D^2)| p)]=\partial w.
\end{equation}
for some $w\in (1+\jmath)\hat{C}_{\bullet}(Y,\spin)$. As $\langle Q,Q^2,Q\rangle$ is zero at the chain level (see Example \ref{chainvan}), we have then that $\langle \x, Q,Q^2,Q\rangle$ is represented by
\begin{equation*}
\hat{m}_2(w| S^1)+\hat{m}_4(x|S^1|p\cup \jmath p|S^1)+\hat{m}_3(t|p\cup \jmath p|S^1).
\end{equation*}
We claim that under the inclusion into $\hat{C}_{\bullet}(Y,\spin)$ this maps to the same class as $\hat{m}(t+x\times D^2| p)$, so that the result follows.
We have
\begin{align*}
&\hat{m}_4(x|S^1| p|S^1)=\\
&=\partial \hat{m}_4(x|S^1|p|D^2)+\hat{m}_2(m_3(x|S^1|p)| D^2)+\hat{m}_3(\hat{m}_2(x| S^1)| p| D^2)+\hat{m}_2(x| \hat{m}_3(S^1|p|D^2)),
\end{align*}
where as above $\hat{m}_2(p,D^2)=0$, and
\begin{equation*}
\hat{m}_3(t|p| S^1)=\partial \hat{m}_3(t|p| D^2) +\hat{m}_2(\hat{m}_2(t| p)| D^2)+\hat{m}_3(\hat{m}_2(x| S^1)| p|D^2).
\end{equation*}
Notice that, regarding the last term in the first sum,  $m_3(S^1,p,D^2)$ is a collection of points, so that $1+\jmath$ of the last term represents a multiple of $Q^2 \cdot \x=0$. Hence the image $\iota_*(\langle \x, Q,Q^2,Q\rangle)\in \HMf_{\bullet}(Y,\spin)$ is represented by
\begin{equation*}
\hat{m}_2(w|S^1)+(1+\jmath)[\hat{m}_2(\hat{m}_3(x|S^1|p)| D^2)+\hat{m}_2(\hat{m}_2(t |p)| D^2)].
\end{equation*}
Notice that
\begin{equation*}
\hat{m}_2(w| S^1)=\hat{m}_2(w| \partial D^2)=\partial\hat{m}_2(w| D^2)+\hat{m}_2\left( [(1+\jmath)(\hat{m}_2(t+\hat{m}_2(x| D^2)| p)]| D^2\right)
\end{equation*}
and
\begin{align*}
\hat{m}_2(\hat{m}_3(x|S^1|p)|D^2)&=\hat{m}_2(\hat{m}_3(x|\partial D^2|p)|D^2)=\\
&=\hat{m}_2(\partial \hat{m}_3(x|D^2|p)+\hat{m}_2(\hat{m_2}(x| D^2)| p)|D^2)
\end{align*}
Hence $\iota$ of our cycle has the form (up to boundaries)
\begin{equation*}
\hat{m}_2((1+\jmath)\alpha| D^2)+(1+\jmath)\hat{m}_2(\alpha+\partial \beta|D^2)
\end{equation*}
where $\alpha=\hat{m}_2(t+\hat{m}_2(x| D^2)| p)$, or equivalently
\begin{equation*}
\hat{m}_2(\jmath\alpha| D^2\cup \jmath D^2) +\partial [(1+\jmath)\hat{m}_2(\beta| D^2)].
\end{equation*}
Now course $\hat{m}_2(\jmath\alpha| D^2\cup \jmath D^2)$ represents the same class as $\jmath\alpha$. Finally, the relation (\ref{fp}) implies that this is also the class of $\alpha$ itself, and the result follows.
\end{proof}
\vspace{0.3cm}
Let us point out some immediate consequences regarding correction terms. There is a plethora of numerical invariants of homology cobordism of rational homology spheres that one can in principle extract from the $\Rin$-module structure in $\Pin$-monopole Floer homology (these might go under the name \textit{generalized correction terms}). On the other hand, the simplest correction term, namely the Froysh\o v invariant $\delta(Y)$ arising from usual monopole Floer homology cannot be recovered from the $\Rin$-module structure (see \cite{Sto3}). We briefly recall the definition of the latter (which is $-h(Y)$ in the notation of \cite{KM}). We have that, given $Y$ a rational homology sphere,
\begin{equation*}
\HMb_{\bullet}(Y,\spin)\cong \ztwo[U^{-1},U]]
\end{equation*}
and that the map $i_*:\HMb_{\bullet}(Y,\spin)\rightarrow\HMt_{\bullet}(Y,\spin)$ is an isomorphism in degrees high enough and vanishes in degrees low enough. We then define
\begin{equation*}
\delta(Y,\spin)=\frac{1}{2}\mathrm{min}\{\mathrm{deg}(\x)|\x\in\mathrm{Im}(i_*)\}.
\end{equation*}
We then have the following characterization of $\delta(Y,\spin)$ purely in terms of the $\Ainf$-structure in $\Pin$-monopole Floer homology.
\begin{prop}\label{froy}
Let $\x$ be the bottom element of the $\gamma$-tower which is not in the image of $Q$. Then
\begin{equation*}
\delta(Y,\spin)=
\begin{cases}
\frac{1}{2}(\mathrm{deg}(\x)-2)\text{ if either }Q^2\cdot\x\neq 0\text{ or }\langle Q,Q^2,\x\rangle\neq 0\\
\frac{1}{2}\mathrm{deg}(\x)\textit{ otherwise.}
\end{cases}
\end{equation*}
\end{prop}
\begin{proof}
Let $\x$ be the bottom of the element in the $\gamma$-tower which is not in the image of $Q$. By exactness of the Gysin sequence, $\iota(\x)\neq 0$. Furthermore, by comparing with the Gysin sequence in degrees high enough, wee see that $\iota(\x)$ is in fact an element of the $U$-tower of $\HMt_{\bullet}(Y,\spin)$. We claim that $U^2\iota(\x)=0$. In fact, we would have otherwise
\begin{equation*}
0\neq U^2\iota(\x)=V\cdot \iota(\x)=\iota(V\x),
\end{equation*}
so that $V\x$, which is an element in the $\gamma$-tower, is not in the image of $Q$, which contradicts our choice of $\x$. The bottom of the $U$-tower is therefore given by either $\iota(\x)$ or $U\cdot\iota(\x)$. In the latter case, by exactness we have that either $U\cdot\iota(\x)$ is in the image of $\iota$, or its image under $\pi_*$ is non-vanishing. These two cases correspond, thanks to Theorem \ref{Massey} and Lemma \ref{Q2}, to respectively $\langle Q,Q^2,\x\rangle\neq 0$ or $Q^2\cdot\x\neq 0$, and the result follows.
\end{proof}

In fact, a more natural correction term to study in $\Pin$-monopole (introduced in \cite{Sto3}) is 
\begin{equation*}
\delta'(Y,\spin)=\frac{1}{2}(\mathrm{min}\{\mathrm{deg}(\x)|\x\in\gamma\text{-tower}\, \x\not\in\mathrm{Im}Q\}-2)
\end{equation*}
which, by the discussion above, coincides with either $\delta(Y,\spin)$ or $\delta(Y,\spin)+1$. Furthermore, $\delta'(Y,\spin)$ reduces modulo $2$ to $-\mu(Y,\spin)$. While we have
\begin{equation*}
\delta(-Y,\spin)=-\delta(Y,\spin)
\end{equation*}
the effect of orientation reversal on $\delta'(Y,\spin)$ cannot be described purely in term of the module structure. On the other hand, it can be described in terms of Massey products as follows.
\begin{prop}
Let $\x$ be the bottom element of the $\gamma$-tower such that either $Q^2\cdot\x\neq 0$ or $\langle \x,Q^2,Q\rangle=0$, and set
\begin{equation*}
\delta''(Y)=\frac{1}{2}(\mathrm{deg}(\x)-2).
\end{equation*}
Then we have $\delta'(-Y)=-\delta''(Y)$.
\end{prop}
\begin{proof}
The key observation here is that if $\y\in \HSt_{\bullet}(Y)$ has $Q\cdot\y$ in the $\gamma$-tower, then it image $j_*(\y)$ in $\HSf_{\bullet}(Y)$ (which is annihilated by $Q$) is such that $\langle j_*(\y),Q,Q^2\rangle$ is $Q^2\cdot\V$-based. To see this, we notice that exactness implies that $\y$ is not in the image of $\pi_*$, while $j_*(\y)$ is. Comparing this with the exact triangle relating the three flavors of usual monopole Floer homologies, we see by a simple diagram chasing then that $j_*(\y)$ is in the image of a non $U$-torsion element $\z\in \HMf_{\bullet}(Y,\spin)$. For such a $\z$, we have that $\pi_*(U\cdot\z)$ is $Q^2\cdot\V$-based (again by a simple diagram chasing), and the claim follows by Theorem \ref{Massey}. Finally, our main result follows from Poincar\'e duality. \end{proof}

\vspace{0.3cm}
\section{Examples}\label{examples}
In this section we discuss several examples in which the description of the Massey products in terms of the Gysin exact triangle is very explicit.
\\
\par
\textit{Manifolds of simple type $M_n$. }We introduce a special class of homology spheres which play a central role in $\Pin$-monopole Floer homology. Let us first introduce the relevant algebraic definitions. The definition is slightly different according to whether $n$ is even or odd.
\par
In the case $n=2k$, we define
\begin{equation*}
M_n=\ztwo[[V]]\oplus\ztwo[[V]]\langle4k-1\rangle \oplus\ztwo[[V]] \langle4k-2\rangle\oplus \ztwo[[V]]/(V^{k-1})\langle4k-3\rangle
\end{equation*}
where the action of $V$ respects the direct sum decomposition and the action of $Q$ maps one column to the one on the right and has maximal possible rank. This module can be depicted graphically as

\begin{center}
\begin{tikzpicture}
\matrix (m) [matrix of math nodes,row sep=0.1em,column sep=0.7em,minimum width=0.1em]
  {   \ztwo_{4k-1} & \ztwo &\cdot&& \cdots &\cdots& \ztwo & \ztwo & \cdot & \ztwo_0 & \ztwo &\ztwo&\cdots \\
        &  & \ztwo & &  &&&  & \ztwo &  &  & &\\
};
\path[-stealth]
(m-1-2) edge[bend right] node {}(m-2-3)
(m-1-1) edge[bend left] node {}(m-1-5)
(m-1-2) edge[bend left] node {}(m-1-6) 
(m-1-1) edge[bend right] node {}(m-1-2)

(m-1-7) edge[bend right] node {}(m-1-8)
(m-1-10) edge[bend right] node {}(m-1-11)
(m-1-11) edge[bend right] node {}(m-1-12)
(m-1-8) edge[bend right] node {}(m-2-9)
(m-1-8) edge[bend left] node {}(m-1-12)
(m-1-7) edge[bend left] node {}(m-1-11)
;
\draw[decoration={brace,mirror,raise=-4pt,amplitude=15pt},decorate]
  (-5,-0.5) -- node[below=6pt] {$k$ copies} (2,-0.5);

\end{tikzpicture}
\end{center}
where the arrows in the upper and lower rows represent respectively the actions of $V$ and $Q$, and the bottom row corresponds to the $\ztwo[[V]]/(V^{k-1})\langle4k-3\rangle$ summand. Let $1$ be the generator of the first summand and $qv^{-k}$ be the generator of the second summand. They lie in degrees respectively $0$ and $4k-1$, and are $\V$ and $Q\cdot\V$-based.
\par
In the case $n=2k+1$ for $k\geq 0$, we define
\begin{equation*}
M_n=\ztwo[[V]]\langle-2\rangle\oplus \ztwo[[V]]\langle 4k+1\rangle\oplus \ztwo[[V]]\langle 4k\rangle\oplus\ztwo[[V]]/(V^{k-1})\langle 4k-1\rangle.
\end{equation*}
where again the action of $V$ respects the direct sum decomposition and the action of $Q$ maps one column to the one on the right and has maximal possible rank. More visually,
\begin{center}
\begin{tikzpicture}
\matrix (m) [matrix of math nodes,row sep=0.1em,column sep=0.7em,minimum width=0.1em]
  {   \ztwo_{4k+1} & \ztwo &\cdot&& \cdots &\cdots& \ztwo & \ztwo & \cdot &\cdot & \ztwo &\ztwo_0&\cdot&\ztwo&\ztwo&\ztwo\cdots \\
        &  & \ztwo & &  &&&  & \ztwo &  &  & &\\
};
\path[-stealth]
(m-1-2) edge[bend right] node {}(m-2-3)
(m-1-1) edge[bend left] node {}(m-1-5)
(m-1-2) edge[bend left] node {}(m-1-6) 
(m-1-1) edge[bend right] node {}(m-1-2)

(m-1-7) edge[bend right] node {}(m-1-8)
(m-1-11) edge[bend right] node {}(m-1-12)
(m-1-8) edge[bend right] node {}(m-2-9)
(m-1-8) edge[bend left] node {}(m-1-12)
(m-1-7) edge[bend left] node {}(m-1-11)

(m-1-14) edge[bend right] node {}(m-1-15)
(m-1-15) edge[bend right] node {}(m-1-16)

(m-1-11) edge[bend left] node {}(m-1-15)
(m-1-12) edge[bend left] node {}(m-1-16)
;
\draw[decoration={brace,mirror,raise=-4pt,amplitude=15pt},decorate]
  (-6.5,-0.5) -- node[below=6pt] {$k$ copies} (0.5,-0.5);

\end{tikzpicture}
\end{center}

We denote the generator of the first summand by $v$ and the generator of the second summand by $qv^{-k}$. They lie in degrees $-2$ and $4k+1$ respectively, and are $\V$ and $Q\cdot \V$-based.
\\
\par
The following is the key definition of this section (an analogous concept of manifolds of projective type was introduced in \cite{Sto}).
\begin{defn}\label{simple}
A homology sphere $Y$  has $\Pin$-\textit{simple type} $M_n$ if there is a direct sum decomposition as $\Rin$-modules
\begin{equation*}
\HSf_{\bullet}(Y)=M_n\langle-1\rangle \oplus J
\end{equation*}
with $p_*(J)=0$ and no non-trivial Massey products between the two summands.
\end{defn}
From the Gysin exact sequence it readily follows that the part of $\HMf_{\bullet}(Y)$ interacting with $M_n\langle-1\rangle$ has the form
\begin{equation*}
\ztwo[[U]]\langle-1\rangle\oplus \ztwo[[U]]/U^n\langle 2n-2\rangle
\end{equation*}
In particular, if $Y$ has simple type $M_n$, $\delta(Y)=0$ $\alpha(Y)=\beta(Y)=n$ and 
\begin{equation*}
\gamma(Y)=\begin{cases}
0 \text{ if $n$ is even}\\
1 \text{ otherwise.}
\end{cases}
\end{equation*}
In particular, the Rokhlin invariant of $Y$ is simply the parity of $n$.
\\
\par
\textit{Surgery on $L$-space knots. }There are several examples of manifolds with simple type $M_n$ obtained by surgery on a knot in $S^3$ (this should be compared, in the Heegaard Floer setting, to the results in \cite{HHL}). In this subsection, we prefer to work for notational reasons with the \textit{to} homologies $\HSt_{\bullet}$. Let us introduce a notation for the standard $U$ and $V$-towers
\begin{align*}
\U^+&=\ztwo[U^{-1},U]]/\ztwo[[U]]\\
\V^+&=\ztwo[V^{-1},V]]/\ztwo[[V]],
\end{align*}
where the bottom element lies in degree zero. Recall that a knot $K\subset S^3$ is called an \textit{$L$-space knot} if for large enough $r>0$, the manifold $S^3_r(K)$ is an $L$-space (i.e. it has vanishing reduced monopole Floer homology $\mathit{HM}$). Given an $L$-space knot, we have
\begin{equation}\label{Lspaceknotn}
\HMt_{\bullet}(S^3_0(K),\spin_0)=\U^+\langle-1\rangle \oplus \U^+\langle-2n\rangle
\end{equation}
for some $n\geq0$. Here $\spin_0$ denotes the unique torsion spin$^c$ structure. We say in this case that $K$ has \textit{type} $n$. The analogous fact in Heegaard Floer homology is well-known \cite{OS2}, and implies our claim via the isomorphism with monopole Floer homology (see \cite{HFHM1}, \cite{CGH1}, and subsequent papers). Our key source of examples is the following.

\begin{prop}\label{manysimple}
Let $K$ be an $L$-space knot of type $n$. Then the manifold $-S^3_{-1}(K)$ has simple type $M_n$, where $-Y$ denotes $Y$ with the orientation reversed.
\end{prop}
\begin{example}
Recall that all positive torus knots have a positive lens space, hence $L$-space, surgery. In particular, the torus knot $T(2,4n-1)$ is an $L$-space knot of type $n$, and furthermore $S^3_{-1}(T(2,4n-1))$ is the Seifert space $\Sigma(2,4n-1, 8n-1)$. Therefore, $-\Sigma(2,4n-1, 8n-1)$ has simple type $M_n$.
\end{example}

Notice that the Arf invariant of an $L$-space knot $K$ is the same as the parity of its type $n$. The proof of Proposition \ref{manysimple} is simpler in the case of $\mathrm{Arf}=1$, and essentially follows from the computations involving the surgery exact triangle in \cite{Lin2}. The proof of the $\mathrm{Arf}=0$ case is more subtle, and follows from the content of the unpublished note \cite{Lin6}. We here review the main ideas involved for both cases. We denote by $\check{S}_{1/q}$ the group $\HSt_{\bullet}(S^3_{1/q}(K))$. For each $q$, the main result of \cite{Lin2} implies that there is an exact triangle
\begin{center}
\begin{tikzpicture}
\matrix (m) [matrix of math nodes,row sep=2em,column sep=1.5em,minimum width=2em]
  {
  \check{S}_{1/q} && \check{S}_{0}\\
  &\check{S}_{1/(q+1)} &\\};
  \path[-stealth]
  (m-1-1) edge node [above]{$\check{A}_q$} (m-1-3)
   (m-2-2) edge node {}(m-1-1)
  (m-1-3) edge node [right]{$\check{B}_{q+1}$} (m-2-2)  
  ;
\end{tikzpicture}
\end{center}
where the maps $\check{A}^s_q$ and $\check{B}^s_q$ are those induced by the spin cobordisms induced by handle attachments. The key observation is the following.
\begin{lemma}\label{key}
The composite
\begin{equation}\label{composite}
\check{B}_{q}\circ \check{A}_q:\check{S}_{1/q} \rightarrow \check{S}_{1/q} 
\end{equation}
is given by multiplication by $Q$. 
\end{lemma}
On the other hand, recall that the analogous map in the usual setting of monopole Floer homology vanishes, see \cite{KMOS}.
\begin{proof}
The composition of the two cobordisms is described by the Kirby diagram in Figure \ref{composition2}. The cobordism from $Y_{1/q}(K)$ to $Y_0(K)$ is given by a two handle attachment along the knot $K'$, while the following one from $Y_0(K)$ to $Y_{1/q}(K)$ is given by attaching another two handle along a zero framed meridian of $K'$. If we trade this second handle for a $1$-handle (i.e. adding a dot in the notation of \cite{GS}), we obtain a pair of canceling $1$ and $2$-handles. In particular, the composite cobordism is obtained from the product one $[0,1]\times Y_{1/q}(K)$ by removing a neighborhood $S^1\times D^3$ of a loop and replacing it by $S^2\times D^2$. The result then readily follows from the fact that the map induced by $S^2\times S^2$ with two ball removed induces multiplication by $Q$ (see the proof of Theorem $5$ in \cite{Lin2}).
\end{proof}

\begin{figure}
  \centering
\def\svgwidth{0.9\textwidth}
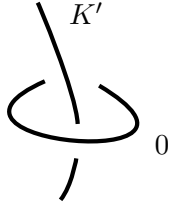
\caption{A handlebody description of the composite of the cobordisms defining the map $\check{B}_{q-1}\circ \check{A}_q$. This link is inside $Y_{1/q}$.}
    \label{composition2}
\end{figure}  As $K$ has Arf invariant zero, we have from \cite{Lin2}, or more generally \cite{Lin5}, that
\begin{equation}
\bar{S}_0= \Rint\langle-1\rangle\oplus\Rint,
\end{equation}
where we fix the identification so that so that $\bar{A}^s_0$ is an isomorphism onto the first summand, so that the triangle looks schematically like
\begin{center}
\begin{tikzpicture}
\matrix (m) [matrix of math nodes,row sep=0.5em,column sep=0.5em,minimum width=2em]
  {

  \ztwo &&&&& \cdot& \ztwo&&&&&\ztwo\\
  \ztwo &&&&& \ztwo&\ztwo&&&&&\ztwo \\
  \ztwo &&&&& \ztwo&\ztwo&&&&& \ztwo\\
  \cdot &&&&& \ztwo&\cdot \\
  };
  \path[-stealth]
(m-1-1) edge node {}(m-2-6) 
(m-2-1) edge node {}(m-3-6) 
(m-3-1) edge node {}(m-4-6) 

(m-1-7) edge node {}(m-1-12)
(m-2-7) edge node {}(m-2-12)
(m-3-7) edge node {}(m-3-12)

  ;
    \draw[dotted]  (-1.5,-1.5) -- (-1.5,1.5);
        \draw[dotted]  (1.5,-1.5) -- (1.5,1.5);
\end{tikzpicture}
\end{center}
repeated in both directions four-periodically. Here the three columns denote $\bar{A}^s_{\infty},\bar{A}^s_0$ and $\bar{A}^s_1$ respectively. Lemma \ref{key} implies that for $q=-1$ (or, in general, $q$ odd) the triangle looks like
\begin{center}
\begin{tikzpicture}
\matrix (m) [matrix of math nodes,row sep=0.5em,column sep=0.5em,minimum width=2em]
  {

  \ztwo &&&&& \cdot& \ztwo&&&&&\ztwo\\
  \ztwo &&&&& \ztwo&\ztwo&&&&&\ztwo \\
  \ztwo &&&&& \ztwo&\ztwo&&&&& \ztwo\\
  \cdot &&&&& \ztwo&\cdot \\
  };
  \path[-stealth]
(m-1-1) edge node {}(m-2-6) 
(m-2-1) edge node {}(m-3-6) 
(m-3-1) edge node {}(m-4-6)
(m-1-1) edge node {}(m-2-7) 
(m-2-1) edge node {}(m-3-7) 

(m-2-6) edge[bend right] node {}(m-2-12)
(m-3-6) edge[bend right] node[bend right] {}(m-3-12)

(m-1-7) edge node {}(m-1-12)
(m-2-7) edge node {}(m-2-12)
(m-3-7) edge node {}(m-3-12)

  ;
    \draw[dotted]  (-1.5,-1.5) -- (-1.5,1.5);
        \draw[dotted]  (1.5,-1.5) -- (1.5,1.5);
\end{tikzpicture}
\end{center}
repeated in both directions four-periodically.
\\
\par
The final observation is that the isomorphism (\ref{Lspaceknotn}) implies via the Gysin exact sequence that, setting
\begin{equation*}
\Rin^+=\HSt_{\bullet}(S^3)=\V^+\oplus Q\cdot\V^+\oplus Q^2\V^+,
\end{equation*}
we have
\begin{equation*}
\HSt_{\bullet}(S^3_0(K))=\Rin^+\langle-1\rangle\oplus \Rin^+\langle-2n\rangle.
\end{equation*}
With this in hand, the proof of Proposition \ref{manysimple} in the Arf $0$ case follows as in the next example.
\begin{example}
Consider the case of the torus knot $K=T(2,7)$, which has is an $L$-space knot of type $2$. We have then the identification of the triangle with $q=-1$ with
\begin{center}
\begin{tikzpicture}
\matrix (m) [matrix of math nodes,row sep=0.5em,column sep=0.5em,minimum width=2em]
  {

  \ztwo &&&&& \cdot& \ztwo&&&&&\ztwo\\
  \ztwo &&&&& \ztwo&\ztwo&&&&&\ztwo \\
  \ztwo &&&&& \ztwo&\ztwo&&&&& \ztwo\\
  \cdot &\ztwo&&&& \ztwo&\cdot \\
  \ztwo&&&&&& \ztwo\\
  \ztwo&&&&&&\ztwo\\
  &&&&&&\ztwo\\
  };
  \path[-stealth]
(m-1-1) edge node {}(m-2-6) 
(m-2-1) edge node {}(m-3-6) 
(m-3-1) edge node {}(m-4-6)
(m-1-1) edge node {}(m-2-7) 
(m-2-1) edge node {}(m-3-7) 

(m-4-2) edge node {}(m-5-7) 
(m-4-2) edge[dotted, bend left] node {}(m-5-1) 
(m-5-1) edge node {}(m-6-7) 
(m-6-1) edge node {}(m-7-7)

  ;
    \draw[dotted]  (-1,-1.5) -- (-1,1.5);
        \draw[dotted]  (1.7,-1.5) -- (1.7,1.5);
                \draw[dotted]  (-3,0.4) -- (3,0.4);
\end{tikzpicture}
\end{center}
where for simplicity we have only depicted $\check{A}_0$, and we have omitted summands in $\HSt_{\bullet}(S^3_{-1}(K))$ arising from the not torsion spin$^c$ structures on $S^3_0(K)$, as they do not have interesting Massey products due to naturality. Here the dotted arrow denotes a $Q$ action, and the horizontal dotted line represents grading zero.. The result then follows from Poincar\'e duality.
\end{example}

The case in which $K$ has Arf invariant one is significantly simple. Indeed, in this case we have
\begin{equation*}
\HSb_{\bullet}(S^3_0(K))=(\V\otimes \ztwo[Q]/Q^2)\oplus(\V\otimes \ztwo[Q]/Q^2)\langle-2\rangle.
\end{equation*}
so that the maps on $\HSb_{\bullet}$ in the exact triangle are uniquely determined, see \cite{Lin2} for the details.
\begin{remark}
In fact, we see that $\HSb_{\bullet}(S^3_0(K))$ has many interesting Massey products itself. The Gysin exact triangle looks in this case like
\begin{center}
\begin{tikzpicture}
\matrix (m) [matrix of math nodes,row sep=1em,column sep=0.5em,minimum width=2em]
  {

  \ztwo &&&&& \cdot& \ztwo&&&&&\ztwo\\
  \ztwo &&&&& \ztwo&\cdot&&&&&\ztwo \\
  \ztwo &&&&& \cdot&\ztwo&&&&& \ztwo\\
  \ztwo &&&&& \ztwo&\cdot&&&&&\ztwo \\
  };
  \path[-stealth]
(m-1-1) edge node {}(m-1-7) 
(m-3-1) edge node {}(m-3-7) 

(m-2-6) edge node {}(m-2-12)
(m-4-6) edge node {}(m-4-12)	

(m-1-1) edge[dotted, bend right] node {}(m-2-1) 
(m-3-1) edge[dotted, bend right] node {}(m-4-1) 
(m-1-12) edge[dotted, bend left] node {}(m-2-12) 
(m-3-12) edge[dotted, bend left] node {}(m-4-12) 
(m-2-6) edge[dotted, bend right] node {}(m-4-6) 
(m-1-7) edge[dotted, bend left] node {}(m-3-7)

  ;
    \draw[dotted]  (-1.5,-1.7) -- (-1.5,1.7);
        \draw[dotted]  (1.5,-1.7) -- (1.5,1.7);
\end{tikzpicture}
\end{center}
where the solid arrows depict the maps $\iota_*$ and $\pi_*$, while the dotted arrows denote the actions of $Q$ and $U$. In light of Theorem \ref{Massey}, we see that there are many non-trivial Massey products of the form $\langle \cdot, Q,Q^2\rangle$ and $\langle\cdot, Q^2,Q\rangle$.
\end{remark}

\vspace{0.3cm}
\textit{Manifolds of simple type $-M_n$. }Of course, from the view point of the Massey products we are interested in, manifolds of simple type $M_n$ are not particularly interesting are there no interesting elements annihilated by elements in $\Rin$. On the other hand, the manifolds obtained by orientation reversal, to which we refer as manifold with simple type $-M_n$, have a richer structure. For simplicity we will focus on the case in which $n$ is even. Let us start from $n=2$. The $\Rin$-module structure for a manifold of type $-M_2$ is given by
\begin{equation*}
\ztwo[[V]]\langle-4\rangle\oplus\ztwo[[V]]\langle -5\rangle\oplus\ztwo[[V]]\langle-2\rangle\oplus\ztwo_0
\end{equation*}
where the action of $Q$ is injective from the first tower to the second, and from the second to the third. Graphically,
\begin{center}
\begin{tikzpicture}
\matrix (m) [matrix of math nodes,row sep=0.1em,column sep=0.7em,minimum width=0.1em]
  {   \ztwo_0 &\cdot  &\ztwo_{-2}&\cdot& \ztwo &  \ztwo &\ztwo&\cdots \\};

\path[-stealth]
(m-1-3) edge[bend left] node {}(m-1-7)
(m-1-5) edge[bend right] node {}(m-1-6)
(m-1-6) edge[bend right] node {}(m-1-7)
;
\end{tikzpicture}
\end{center}
As a notation, we will denote the generator in degree zero by $\z$, while the generators in degree $-2$ and $-4$ by $q^2$ and $v$ respectively. For a general manifold $Y$ of simple type $-M_2$, we will have a decomposition $\HSb_{\bullet}(Y)=-M_2\oplus C$ as $\Rin$-modules, where furthermore $C$ there are no non-trivial Massey-products between the summands.
The key observations is the following.
\begin{lemma}\label{-Mn}
We have the triple Massey products $\langle \z, Q, Q^2\rangle=q^2$ and $\langle V, \z, Q\rangle=v$.
\end{lemma}
\begin{proof}
We use the Gysin sequence characterization of the triple Massey product, Theorem \ref{Massey}. In particular, the corresponding component in $\HMf_{\bullet}$ is given by $\ztwo[[U]]\oplus\ztwo[[U]]/U^2$. In degrees $\geq-6$, the Gysin sequence looks like
\begin{center}
\begin{tikzpicture}
\matrix (m) [matrix of math nodes,row sep=0.5em,column sep=0.5em,minimum width=2em]
  {

  \ztwo &&&&& \ztwo& \ztwo&&&&&\ztwo\\
  \cdot &&&&&\cdot &\cdot&&&&&\cdot \\
  \ztwo &&&&& \ztwo&\ztwo&&&&& \ztwo\\
  \cdot &&&&& &\cdot&&&&&\cdot \\
  \ztwo&&&&&&\ztwo&&&&&\ztwo\\
  \ztwo&&&&&&\cdot&&&&&\ztwo\\
  \ztwo&&&&&&\ztwo&&&&&\ztwo\\
  };
  \path[-stealth]
(m-1-7) edge[] node {}(m-1-12)
(m-3-7) edge[] node {}(m-3-12)
(m-7-7) edge[] node {}(m-7-12)
(m-1-1) edge[] node {}(m-1-6)
(m-3-1) edge[] node {}(m-3-6)
(m-5-1) edge[] node {}(m-5-7)
(m-1-6) edge[bend right, dotted] node {}(m-3-6)
(m-1-7) edge[bend left, dotted] node {}(m-3-7)
(m-3-7) edge[bend left, dotted] node {}(m-5-7)
(m-5-7) edge[bend left, dotted] node {}(m-7-7)
	  
  ;
    \draw[dotted]  (-1.5,-1.7) -- (-1.5,1.7);
        \draw[dotted]  (1.5,-1.7) -- (1.5,1.7);
\end{tikzpicture}
\end{center}
where the dotted arrows represent the $U$ action. Consider the element $1$ in the first summand. Then comparing with the Gysin triangle in $\HSb_{\bullet}$, we see that $\pi_*(1)=\z$. Furthermore $V\iota(q^2)=\iota (q^2v)=0$, hence $U$ is not in the image of $\iota$ and $\pi_*(U)=q^2$. The second statement follows in the same way.
\end{proof}
In the case for general even $n=2k$, we have that the module structure is given by
the $\Rin$-module structure for a manifold of type $-M_1$ is given by
\begin{equation*}
(\ztwo[V]\langle-4k\rangle\oplus\ztwo[V]\langle -4k-1\rangle)\oplus\ztwo[V]\langle-2\rangle\oplus\ztwo_0[V]/V^n
\end{equation*}
where the action of $Q$ is injective from the first tower to the second, and from the second to the third. We will denote the direct sum of the first three terms $N_k$. Graphically, in the case $n=2$,
\begin{center}
\begin{tikzpicture}
\matrix (m) [matrix of math nodes,row sep=0.1em,column sep=0.7em,minimum width=0.1em]
  {   \ztwo_0 &\cdot  &\ztwo_{-2}&\cdot& \ztwo &  \cdot &\ztwo&\cdot &\ztwo&\ztwo&\ztwo&\cdots\\};

\path[-stealth]
(m-1-3) edge[bend left] node {}(m-1-7)
(m-1-7) edge[bend left] node {}(m-1-11)
(m-1-9) edge[bend right] node {}(m-1-10)
(m-1-10) edge[bend right] node {}(m-1-11)
(m-1-1) edge[bend left] node {}(m-1-5)
;
\end{tikzpicture}
\end{center}
Denote the generators of the summands by $q^2$, $v^n$, $qv^n$ and $\z$. We then have for example the relations
\begin{equation*}
\langle \z,Q,Q^2\rangle=q^2\quad \langle V^n,\z,Q\rangle=v^n.
\end{equation*}
This follows in the same way as Lemma \ref{-Mn}, using the fact that the corresponding component in $\HMf_{\bullet}$ is given by $\ztwo[U]\oplus\ztwo[U]/U^{2k}$.

\vspace{0.3cm}
\textit{Seifert spaces. }Another large class of manifolds for which the Massey products described in Theorem \ref{Massey} can be understood explicitly is given by Seifert spaces. The main observation here is that we can assume that all irreducible solutions have odd degree for a suitable choice of orientation (see \cite{MOY}, and also \cite{Dai} for a discussion of the more general case of plumbed manifolds).
\begin{example}\label{mccoy}
Rather that describing a general theory (which would be analogous to parts of the content of \cite{DaiS}), let us focus on an example (due to Duncan McCoy) that involves three or more summands. Consider the Seifert space $Y=\Sigma(13,21,34)$. Then, up to grading shifts, we have that
\begin{equation*}
\HMf_{\bullet}(Y)=\ztwo[[U]]\oplus(\ztwo[[U]]/U^6)_{11}\oplus(\ztwo[[U]]/U^5)_{11} \oplus(\ztwo[[U]]/U^4)_{9}\oplus J^{\oplus 2}
\end{equation*}
where the involution action exchanges the two copies of $J^{\otimes 2}$. As for this orientation there are only irreducible critical points of odd degree, it is straightforward to reconstruct the underlying chain complex $\hat{C}_{\bullet}(Y)$ (where we forget about the $J^{\oplus 2}$ summand as it is irrelevant for our purposes):
\begin{center}
\begin{tikzpicture}
\matrix (m) [matrix of math nodes,row sep=0.5em,column sep=0.7em,minimum width=0.1em]
  {  &\ztwo &&\ztwo &&\ztwo &&\ztwo &&\ztwo &&\ztwo_0 &\cdot&\ztwo &\\
  \ztwo^2_{\x}&&\ztwo^2&&\ztwo^2&&\ztwo^2&&\ztwo^2&&\ztwo^2\\
  &&\ztwo^2_{\y}&&\ztwo^2&&\ztwo^2&&\ztwo^2\\};

\path[-stealth]
(m-1-2) edge[bend right,dashed] node {}(m-2-3)
(m-1-2) edge[bend right,dashed] node {}(m-3-3)
(m-1-4) edge[bend right,dashed] node {}(m-2-5)
(m-1-4) edge[bend right,dashed] node {}(m-3-5)
(m-1-6) edge[bend right,dashed] node {}(m-2-7)
(m-1-6) edge[bend right,dashed] node {}(m-3-7)
(m-1-8) edge[bend right,dashed] node {}(m-2-9)
(m-1-8) edge[bend right,dashed] node {}(m-3-9)
(m-1-10) edge[bend right,dashed] node {}(m-2-11)
(m-1-2) edge[bend left,dotted] node {}(m-1-4)
(m-1-4) edge[bend left,dotted] node {}(m-1-6)
(m-1-6) edge[bend left,dotted] node {}(m-1-8)
(m-1-8) edge[bend left,dotted] node {}(m-1-10)
(m-1-10) edge[bend left,dotted] node {}(m-1-12)
(m-1-12) edge[bend left,dotted] node {}(m-1-14)
(m-2-1) edge[bend left,dotted] node {}(m-2-3)
(m-2-3) edge[bend left,dotted] node {}(m-2-5)
(m-2-5) edge[bend left,dotted] node {}(m-2-7)
(m-2-7) edge[bend left,dotted] node {}(m-2-9)
(m-2-9) edge[bend left,dotted] node {}(m-2-11)
(m-3-3) edge[bend left,dotted] node {}(m-3-5)
(m-3-5) edge[bend left,dotted] node {}(m-3-7)
(m-3-7) edge[bend left,dotted] node {}(m-3-9)
;
\end{tikzpicture}
\end{center}
Here the first row represents the tower corresponding to the reducible solution, while the second and third rows correspond to the irreducible solutions; the dotted arrows depict the action of $U$, while the dashed ones represent the differential (where each dashed arrow sends the generator of $\ztwo$ to the sum of the generators of $\ztwo^2$). The natural involution $\jmath$ fixes the first row and exchanges the summands in each copy of $\ztwo^2$. Also, we labeled the irreducible generators as $\ztwo[[U]]$-modules by $\x$, $\y$ and their conjugates via $\jmath$. The invariant chain complex $\hat{C}^{\jmath}_{\bullet}(Y)$ is therefore
\begin{center}
\begin{tikzpicture}
\matrix (m) [matrix of math nodes,row sep=0.5em,column sep=0.7em,minimum width=0.1em]
  {  &\ztwo &\ztwo&\ztwo &&\ztwo &\ztwo&\ztwo &&\ztwo &\ztwo&\ztwo_0 &\cdot&\ztwo &\\
  \ztwo&&\underline{\ztwo}&&\ztwo&&\ztwo&&\ztwo&&\ztwo\\
  &&\underline{\ztwo}&&\ztwo&&\ztwo&&\ztwo\\};

\path[-stealth]
(m-1-2) edge[bend right,dashed] node {}(m-2-3)
(m-1-2) edge[bend right,dashed] node {}(m-3-3)

(m-1-6) edge[bend right,dashed] node {}(m-2-7)
(m-1-6) edge[bend right,dashed] node {}(m-3-7)
(m-1-10) edge[bend right,dashed] node {}(m-2-11)
;
\end{tikzpicture}
\end{center}
where the two underlined summands are generated by respectively $(1+\jmath)U\x$ and $(1+\jmath)\y$.

The $\Pin$-monopole Floer homology $\HSf_{\bullet}(Y)$ is then
\begin{center}
\begin{tikzpicture}
\matrix (m) [matrix of math nodes,row sep=0.5em,column sep=0.7em,minimum width=0.1em]
  {  & \cdot&\ztwo&\ztwo_8 &&\cdot &\ztwo&\ztwo &&\cdot &\ztwo&\ztwo_0 &\cdot&\ztwo &\\
  \ztwo&&&&\ztwo&&&&\ztwo&&\\
  &&\underline{\ztwo}&&\ztwo&&\ztwo&&\ztwo\\};

\path[-stealth]

(m-1-4) edge[bend right] node {}(m-2-5)
(m-1-4) edge[bend right] node {}(m-3-5)
(m-1-8) edge[bend right] node {}(m-2-9)
(m-1-8) edge[bend right] node {}(m-3-9)
(m-2-1) edge[bend left] node {}(m-2-5)
(m-2-5) edge[bend left] node {}(m-2-9)
(m-3-3) edge[bend left] node {}(m-3-7)
(m-3-5) edge[bend left] node {}(m-3-9)
(m-2-1) edge[bend right] node {}(m-3-3)
;
\end{tikzpicture}
\end{center}
where the solid arrows denote the $Q$ and $V$ actions. From this description, and the fact that the Gysin exact sequence is the long exact sequence induced by the short exact sequence of chain complexes (\ref{gysinchain}), one can readily determine the non-trivial Massey products. Let us spell out a specific example. Denote the underlined class in $\HSf_{\bullet}(Y)$ by $\z$. It is represented by either $(1+\jmath)U\x$ or $(1+\jmath)\y$. The Massey product $\langle \z,Q^2,Q\rangle$ is by Theorem \ref{Massey} a class mapping via $\iota$ to the class in $\HMf_{\bullet}(Y)$ of either $(1+\jmath)U^2\x$ or $(1+\jmath)U\y$. Each of these generates one of the $\ztwo$ summands in degree $7$. Of course, once we quotient by the image of $Q$, they are identified, as their sum is the image of the class in degree $8$. Even though involving a different definition, the Massey product $\langle \z,Q,Q^2\rangle$ also consists of the same two classes, which are again identified under the image of $Q^2$.
\end{example}

\vspace{0.3cm}
\section{Some homological algebra over $\Rin$}\label{projres}
In this section, we discuss some homological algebra relevant in the description of the $E^2$-page of our spectral sequence
\begin{equation*}
\Tor^{\Rin}_{*,*}(M,N)
\end{equation*} 
for a pair of graded $\Rin$-modules $M$ and $N$. Recall that $\Tor^{\Rin}_{*,*}$ arises in our setting naturally as the homology of the tensor product of $M$ with the bar resolution of $N$. While the latter object has nice formal properties, it is quite unmanageable for actual explicit computations. As the computation of $\Tor^{\Rin}_{*,*}$ is independent of the choice of projective resolution, we first discuss how to compute a particularly nice projective resolution of $N$, called \textit{minimal free resolution}. As the name suggests, this will be very efficient in terms of size. On the other hand, when discussing higher differentials we will need to represent classes in $E^2$ as elements of the bar complex, and the second part of the section will be devoted to translating back in this language the construction using minimal resolutions.
\\
\par
While there is in general no satisfactory classification of finitely generated modules over $\Rin$, the theory of their resolutions is quite well understood, see for example \cite{Eis} and \cite{Ser} for a more general and detailed treatment. The ring $\Rin$ is local with maximal ideal $\m$ generated by $Q$ and $V$. Given a graded module $L$, we will denote by $\bar{L}$ the $\ztwo$-vector space $L/\m L$. We say that a graded $\Rin$-module homomorphism $u:L\rightarrow L'$ is \textit{minimal} if
\begin{itemize}
\item $u$ is surjective;
\item $\mathrm{ker}(u)\subset \m L$.
\end{itemize}
This is equivalent (by Nakayama's lemma) to require that the induced map $\bar{u}:\bar{L}\rightarrow \bar{L}'$ is an isomorphism. A \textit{graded minimal free resolution} of a graded $\Rin$-module $N$ is a graded resolution of the form
\begin{equation*}
N\stackrel{d_0}{\longleftarrow}\Rin^{n_1}\stackrel{d_1}{\longleftarrow}\Rin^{n_2}\stackrel{d_2}{\longleftarrow}\Rin^{n_3}\longleftarrow\cdots
\end{equation*}
where $d_i:\Rin^{n_i}\rightarrow \mathrm{ker}(d_{i-i})\subset \Rin^{n_{i-1}}$ is minimal for each $i$. Here we omit the from the notation the grading shift of each $\Rin$ component. The main result from \cite{Eis} is the following.
\begin{thm}
Every finitely generated $\Rin$-module $N$ admits a graded minimal free resolution. Any two minimal free resolution are (non canonically) isomorphic, and indeed two periodic for $i\geq 3$ (up to grading shift). In particular, we have the isomorphism
\begin{equation*}
\Tor^{\Rin}_{i+2,k-3}(M,N)\cong \Tor^{\Rin}_{i,k}(M,N)
\end{equation*}
for $i\geq 2$.
\end{thm}
For our purposes, the existence statement is the most important, and we now quickly review the explicit construction. Notice first that for any module $M$, there is a minimal map $u:\Rin^n\rightarrow M$. This can be constructed by choosing a basis $\{\bar{e}_i\}$, $i=1,\dots, N$ of $\bar{M}$ over $\ztwo$. Lifting these elements to $e_i\in M$ provides a minimal
\begin{equation*}
u:\Rin^N\rightarrow M.
\end{equation*}
Given now a $\Rin$-module $N$, choose a minimal map $d_0:\Rin^{n_0}\rightarrow N$. Inductively, we can choose a minimal map $d_i:\Rin^{n_i}\rightarrow \mathrm{ker} d_{i-1}$, and these form a minimal resolution.
\par
Let us comment about the rest of the statement. The two-periodicity of the minimal free resolution is a general consequence of the fact that we are considering a hypersurface, namely the zero set of the polynomial $(Q^3)\subset \ztwo[[Q,V]]$, see \cite{Eis} (notice that while several results in the paper do not hold for finite fields, the results about two-periodic resolutions in Section $5$ and $6$ hold). Furthermore, the dimension $n_i$ is independent of $i\geq3$. In particular, to each module $N$ we can associate a matrix factorization
\begin{equation*}
\ztwo[[Q,V]]^n \xtofrom[B]{A}\ztwo[[Q,V]]^n
\end{equation*}
where $A$ and $B$ are the $n\times n$ matrices corresponding to $d_{2i}$ and $d_{2i+1}$ for $i>>0$, respectively. These have the property that $AB=BA=Q^3\cdot I$.
\\
\par
Given this general discussion, let us provide some concrete examples in which we describe the minimal free resolution.
\begin{example}
The trivial $\Rin$-module $\ztwo$ (thought of in degree zero) has a projective resolution
\begin{equation*}
\ztwo\stackrel{d_0}{\longleftarrow} \Rin\stackrel{d_1}{\longleftarrow} \Rin\oplus\Rin\stackrel{d_2}{\longleftarrow} \Rin\oplus\Rin\stackrel{d_3}{\longleftarrow} \Rin\oplus\Rin\stackrel{d_4}{\longleftarrow}\cdots
\end{equation*}
where, in matrix notation, $d_1=(V,Q)$, and for $i\geq 0$ we have
\begin{equation*}
d_i=
\begin{cases}
\begin{bmatrix}
Q & 0\\
V & Q^2
\end{bmatrix}& \text{ if }i\text{ is even}\\
\begin{bmatrix}
Q^2 & 0\\
V & Q
\end{bmatrix}& \text{ if }i\text{ is odd}.
\end{cases}
\end{equation*}
This is clearly $2$-periodic for $i\geq2$.
\end{example}
\begin{example}
Consider the module
\begin{center}
\begin{tikzpicture}
\matrix (m) [matrix of math nodes,row sep=0.1em,column sep=0.7em,minimum width=0.1em]
  { \ztwo&\cdot &\cdot&\ztwo&\ztwo&\cdot&\ztwo&\ztwo&\ztwo&\cdots  \\
  &\ztwo\\};

\path[-stealth]

(m-1-1) edge[bend right] node {}(m-2-2)
(m-1-4) edge[bend right] node {}(m-1-5)
(m-1-8) edge[bend right] node {}(m-1-9)
(m-1-7) edge[bend right] node {}(m-1-8)
(m-1-1) edge[bend left] node {}(m-1-5)
(m-1-4) edge[bend left] node {}(m-1-8)
(m-1-5) edge[bend left] node {}(m-1-9)
;
\end{tikzpicture}
\end{center}
This has the minimal free resolution
\begin{equation*}
N\stackrel{d_0}{\longleftarrow}\Rin\stackrel{d_1}{\longleftarrow}\Rin^{3}\stackrel{d_2}{\longleftarrow}\Rin^{3}\stackrel{d_3}{\longleftarrow}\Rin^{3}\longleftarrow\cdots
\end{equation*}
where 
\begin{equation*}
d_1=\begin{bmatrix}Q &0&0&0\\
V&Q^2&Q&0\\
0&0&V&Q^2\\\end{bmatrix}
\end{equation*}
and for $i\geq 1$ we have
\begin{equation*}
d_{2i}=\begin{bmatrix}Q^2 &0&0&0\\
V& Q&0&0\\
0& 0&Q^2&0\\
0&0&V&Q\\
\end{bmatrix}
\quad
d_{2i+1}=\begin{bmatrix}Q &0&0&0\\
V& Q^2&0&0\\
0& 0&Q&0\\
0&0&V&Q^2\\
\end{bmatrix}\end{equation*}
\end{example}
\begin{example}
Consider the module
\begin{center}
\begin{tikzpicture}
\matrix (m) [matrix of math nodes,row sep=0.1em,column sep=0.7em,minimum width=0.1em]
  { \ztwo&\cdot &\ztwo&\ztwo&\ztwo&\cdots  \\
  &\ztwo\\};

\path[-stealth]

(m-1-1) edge[bend right] node {}(m-2-2)
(m-1-4) edge[bend right] node {}(m-1-5)
(m-1-3) edge[bend right] node {}(m-1-4)
(m-1-1) edge[bend left] node {}(m-1-5)

;
\end{tikzpicture}
\end{center}
This has the minimal free resolution
\begin{equation*}
N\stackrel{d_0}{\longleftarrow}\Rin\stackrel{d_1}{\longleftarrow}\Rin^{3}\stackrel{d_2}{\longleftarrow}\Rin^{3}\stackrel{d_3}{\longleftarrow}\Rin^{3}\longleftarrow\cdots
\end{equation*}
where 
\begin{equation*}
d_1=\begin{bmatrix}Q^2 &0\\
V&Q^2\\\end{bmatrix}
\end{equation*}
and for $i\geq 1$
\begin{equation*}
d_{2i}=\begin{bmatrix}Q^2 &0\\
V& Q\\
\end{bmatrix}
\quad
d_{2i+1}=\begin{bmatrix}Q &0\\
V& Q^2\\
\end{bmatrix}
\end{equation*}

\end{example}
\begin{example}
Consider the module
\begin{center}
\begin{tikzpicture}
\matrix (m) [matrix of math nodes,row sep=0.1em,column sep=0.7em,minimum width=0.1em]
  { \ztwo&\ztwo&\cdot &\cdot&\ztwo  \\};

\path[-stealth]

(m-1-1) edge[bend right] node {}(m-1-2)
(m-1-1) edge[bend left] node {}(m-1-5)
;
\end{tikzpicture}
\end{center}
This has the minimal free resolution
\begin{equation*}
N\stackrel{d_0}{\longleftarrow}\Rin\stackrel{d_1}{\longleftarrow}\Rin^{3}\stackrel{d_2}{\longleftarrow}\Rin^{3}\stackrel{d_3}{\longleftarrow}\Rin^{3}\longleftarrow\cdots
\end{equation*}
where 
\begin{equation*}
d_1=\begin{bmatrix}V^2 &QV&Q^2\end{bmatrix}
\end{equation*}
and for $i\geq 1$ we have
\begin{equation*}
d_{2i}=\begin{bmatrix}Q &0&0\\
V& Q&0\\
0& V&Q
\end{bmatrix}
\quad
d_{2i+1}=\begin{bmatrix}Q^2 &0&0\\
QV& Q^2&0\\
V^2& QV&Q^2
\end{bmatrix}
\end{equation*}
This provides an interesting factorization of $Q^3\cdot I_3$.
\end{example}

\vspace{0.5cm}

We now discuss how to represent classes in $\Tor^{\Rin}_{*,*}(M,N)$ in terms of the bar resolution. Let us start with the simplest case of $\Tor^{\Rin}_{*,*}(\ztwo,\ztwo)$, which is well studied in light of the classical Eilenberg-Moore spectral sequence in algebraic topology (see for example \cite{McC}). In this case, for graded algebras $A$ and $B$ over $\ztwo$, there is an isomorphisms of graded bimodules
\begin{equation*}
\Tor^{A\otimes B}_{*,*}(\ztwo,\ztwo)=\Tor^{A}_{*,*}(\ztwo,\ztwo)\otimes \Tor^{B}_{*,*}(\ztwo,\ztwo)
\end{equation*}
which is indeed also an isomorphism of coalgebras. Denoting the bar resolution by $\hat{B}$, this is induced by the map
\begin{equation*}
\hat{B}(A)\otimes \hat{B}(B)\stackrel{EZ}{\longrightarrow}\hat{B}(A\otimes B).
\end{equation*}
The map $EZ$ is the shuffle map appearing in the proof of the Eilenberg-Zilber theorem, namely
\begin{equation*}
EZ\big((a_1\lvert \cdots\lvert a_p) \otimes (b_1\lvert \cdots\lvert b_q)\big)=\sum_{(p,q)-\text{shuffle }\sigma} c_{\sigma(1)}\lvert \cdots\lvert c_{\sigma(p+q)}
\end{equation*}
where 
\begin{equation*}
c_{\sigma(i)}=\begin{cases}
a_{\sigma(i)}\otimes 1 \text{ if }1\leq \sigma(i)\leq p\\
1\otimes b_{\sigma(i)-p}\text { if }p+1\leq \sigma(i)\leq p+q,
\end{cases}
\end{equation*}
and a $(p,q)$-shuffle is a permutation $\sigma$ of $\{1,\dots, p+q\}$ such that
\begin{align*}
\sigma(1)<\sigma(2)<&\dots <\sigma(p-1)<\sigma(p)\\
\sigma(p+1)<\sigma(p+2)<&\dots <\sigma(p+1-1)<\sigma(p+q).
\end{align*}
In particular, the computation in our case is quite simple:
\begin{equation*}
\Tor^{\Rin}_{*,*}(\ztwo,\ztwo)=\Tor^{\ztwo[[V]]}_{*,*}(\ztwo,\ztwo)\otimes\Tor^{\ztwo[Q]/Q^3}_{*,*}(\ztwo,\ztwo).
\end{equation*}
The group $\Tor^{\ztwo[[V]]}(\ztwo,\ztwo)$ is non zero only in bidegrees $(0,0)$ and $(1,-3)$ (generated by the empty product $[]$ and $V$ respectively), while $\Tor^{\ztwo[Q]/Q^3}(\ztwo,\ztwo)$ is non zero in degrees $(2i,-3i)$ and $(2i+1,-3i-1)$ for $i\geq0$. If we define the $n$-uple
\begin{equation*}
Q_n=\begin{cases}
(Q,Q^2,Q,\cdots, Q,Q^2,Q)&\text{ if }n\text{ is odd,}\\
(Q^2,Q,Q^2,\cdots, Q,Q^2,Q)&\text{ if }n\text{ is even,}
\end{cases}
\end{equation*}
then the generator of $\Tor^{\ztwo[Q]/Q^3}_{*,*}(\ztwo,\ztwo)$ in homological degree $n$ is represented by $Q_n$.
Putting these pieces together, we see that $\Tor^{\Rin}_{*,*}(\ztwo,\ztwo)$ is a copy of $\ztwo$ in the cases in which $(i,j)$ is
\begin{itemize}
\item $(0,0)$;
\item $(2n,-3n)$ and $(2n,-3n-2)$ for $n\geq0$;
\item $(2n+1,-3n-1)$ and $(2n+1,-3n-4)$ for $n\geq0$.
\end{itemize}
The generators of $\Tor^{\Rin}_{*,*}(\ztwo,\ztwo)$ can then be described explicitly in terms of the shuffle map. Given an orderered $n$-uple $(a_1,\dots,a_n)$ and element $b$ of elements in $\Rin$, we define their shuffle as
\begin{equation*}
\mathrm{sh}((a_1,\dots,a_n),b)=\sum_{i=0}^{n} a_1\lvert\cdots \lvert a_i\lvert b\lvert a_{i+1}\lvert\cdots\lvert a_n\in \Rin^{\otimes(n+1)}.
\end{equation*}
Then representatives of the classes described above are given respectively by
\begin{itemize}
\item $[]$;
\item $Q_{2n}$ and $\mathrm{sh}(Q_{2n-1},V)$;
\item $Q_{2n+1}$ and $\mathrm{sh}(Q_{2n},V)$.
\end{itemize}
The following picture represents the groups for $i\leq 7$, where the top right element has bigrading $(0,0)$. It should make apparent the two periodicity of the $\Tor^{\Rin}_{*,*}(\ztwo,\ztwo)$.
\begin{center}
\begin{tikzpicture}
\matrix (m) [matrix of math nodes,row sep=0.1em,column sep=0.7em,minimum width=0.1em]
  { \ztwo\\
  &\ztwo\\
  &\cdot\\
  &\cdot&\ztwo\\
  &\ztwo&\cdot&\ztwo\\
  &&\ztwo&\cdot&\\
  &&&\cdot&\ztwo\\
  &&&\ztwo&\cdot&\ztwo\\
  &&&&\ztwo&\cdot\\
  &&&&&\cdot&\ztwo\\
  &&&&&\ztwo&\cdot\\
  &&&&&&\ztwo\\};
\end{tikzpicture}
\end{center}

\vspace{0.3cm}
For general $\Tor^{\Rin}_{*,*}(M,N)$, we can adapt this approach involving shuffle maps by taking into account a minimal resolution of one of the two modules. Consider the minimal free resolution
\begin{equation*}
N\stackrel{d_0}{\longleftarrow}\Rin^{n_1}\stackrel{d_1}{\longleftarrow}\Rin^{n_2}\stackrel{d_2}{\longleftarrow}\Rin^{n_3}\stackrel{d_3}{\longleftarrow}\Rin^{n_4}\longleftarrow\cdots,
\end{equation*}
so that $\Tor^{\Rin}_{*,*}(M,N)$ is the homology of the chain complex
\begin{equation*}
M^{n_1}\stackrel{1_M\otimes d_1}{\longleftarrow}M^{n_2}\stackrel{1_M\otimes d_2}{\longleftarrow}M^{n_3}\stackrel{1_M\otimes d_3}{\longleftarrow}M^{n_4}\longleftarrow\cdots
\end{equation*}
obtained by tensoring over $\Rin$ with $M$.
Our goal is to define a canonical quasi-isomorphism $\{\varphi_i\}$ between the minimal free resolution
\begin{equation*}
\Rin^{n_1}\stackrel{d_1}{\longleftarrow}\Rin^{n_2}\stackrel{d_2}{\longleftarrow}\Rin^{n_3}\stackrel{d_3}{\longleftarrow}\Rin^{n_4}\longleftarrow\cdots
\end{equation*}
and the bar resolution
\begin{equation*}
\Rin\otimes  N\stackrel{\delta_1}{\longleftarrow} \Rin\otimes\Rin\otimes  N\stackrel{\delta_2}{\longleftarrow} \Rin\otimes\Rin\otimes\Rin\otimes  N\stackrel{\delta_3}{\longleftarrow}\cdots
\end{equation*}
Of course, the map
\begin{equation*}
\varphi_1: \Rin^{n_1}\rightarrow \Rin\otimes N
\end{equation*}
is given by
\begin{equation*}
\x\mapsto 1\otimes d_0(\x).
\end{equation*}
Suppose now inductively that we are given
\begin{equation*}
\varphi_i: \Rin^{n_i}\rightarrow \Rin^{i}\otimes N.
\end{equation*}
We have $d_{i+1}(e_j)=\sum r_{jk} e'_k$ where $e_j$ and $e'_k$ are the standard bases of $\Rin^{n_{i+1}}$ and $\Rin^{n_{i}}$ respectively. We then define
\begin{align*}
\varphi_{i+1}&: \Rin^{n_{i+1}}\rightarrow \Rin^{i+1}\otimes N\\
e_j&\mapsto \sum 1\otimes(r_{jk}\cdot \varphi_i(e'_k)).
\end{align*}
This is readily checked to be a chain map, and as the complexes are acyclic in degrees $\geq 1$ it is a quasi-isomorphism. Using this quasi-isomorphism, one can readily describe elements in $\Tor^{\Rin}_{*,*}$ in terms of a minimal free resolution of $N$. Indeed, if $\underline{m}=(m_j)\in M^{n_k}$ is in the kernel of $d_{k-1}$, then it corresponds to the cycle
\begin{equation*}
\sum m_j \otimes \varphi_{k+1}(e_j)\in M\otimes_{\Rin} \Rin^{k}\otimes N=M\otimes\Rin^{k-1}\otimes N.
\end{equation*}
Let us discuss this rather abstract construction in a very concrete example.
\begin{example}\label{ztwotor}
Let us generalize the description of $\Tor^{\Rin}_{*,*}(\ztwo,\ztwo)$ in terms of shuffles to $\Tor^{\Rin}_{*,*}(M,\ztwo)$, for any $\Rin$-module $M$. We computed above that the minimal free resolution of $\ztwo$ is given by
\begin{equation*}
\ztwo\stackrel{d_0}{\longleftarrow} \Rin\stackrel{d_1}{\longleftarrow} \Rin^2\stackrel{d_2}{\longleftarrow} \Rin^2\stackrel{d_3}{\longleftarrow} \Rin^2\stackrel{d_4}{\longleftarrow}\cdots,
\end{equation*}
and using the description above one can write explicit representatives for all the cycles in $\mathrm{Tor}^{\Rin}_{\ast,\ast}(M,\ztwo)$ as follows. Let us denote by $\z$ the generator of $\ztwo$. For $n\geq 1$, every element of $\mathrm{Tor}^{\Rin}_{n,\ast}(M,\ztwo)$ has a representative of the form
\begin{equation*}
\x\lvert \mathrm{sh}(Q_{n-1},V)\lvert z+ \y\lvert Q_n\lvert z
\end{equation*}
where $d_{n-1}(\x,\y)=0$.
For example:
\begin{itemize}
\item every element in $\mathrm{Tor}^{\Rin}_{1,\ast}(M,\ztwo)$ has a representative of the form
\begin{equation*}
\x\lvert V\lvert \z+\y\lvert Q\lvert \z
\end{equation*}
with $V\x+Q\y=0$, and such an element is zero if and only if $\x=Q\A$ and $\y=V\A+Q^2\B$ for some $\A,\B$;
\item every element in $\mathrm{Tor}^{\Rin}_{2,\ast}(M,\ztwo)$ is represented by
\begin{equation*}
\A \lvert Q\lvert V\lvert \z+\A \lvert V\lvert Q\lvert \z+\A \lvert Q^2\lvert Q\lvert \z
\end{equation*}
where $Q\A=0$ and $V\A+Q^2\B=0$, and such an element is zero if and only if $\A=Q^2\x$ and $\B=V\x+Q\y$ for some $\x,\y$;
\item every element in $\mathrm{Tor}^{\Rin}_{3,\ast}(M,\ztwo)$ is represented by
\begin{equation*}
\x|Q^2|Q|V|\z+\x|Q^2|V|Q|\z+\x|V|Q^2|Q|\z+\y|Q|Q^2|Q|\z,
\end{equation*}
with $Q^2\x=V\x+Q\y=0$, and such an element is zero if and only if $\x=Q\A$ and $\y=V\A+Q^2\B$ for some $\A,\B$.
\end{itemize}
The description then can be carried generalized to $n\geq 4$ in a $2$-periodic fashion.
\end{example}
\vspace{0.3cm}
\section{Connected sums with manifolds of simple type}\label{connsimple}
In this section we study the effect on Floer homology of the connected sum with a manifold of simple type $M_n$ or its opposite $-M_n$, see Definition \ref{simple}. Of course, if we are interested only in the information related to homology cobordism contained in $\HSf_{\bullet}$, we do not need to consider the additional summand $J$ in Definition \ref{simple}. We know from the previous section the general recipe to compute the $E^2$ page of the Eilenberg-Moore spectral sequence, and the goal of this section is to understand higher differentials and extensions. There are several different cases to discuss, and our treatment will combine general results with explicit examples. Before dwelling in our main cases of interest, let us discuss a warm up example.
\begin{example}\label{FF}
Suppose we have a $\Rin$-module decomposition $\HSf_{\bullet}(Y)=M\oplus\ztwo$ with no non-trivial Massey products among the two summands. We want to understand the contribution to $\HSf_{\bullet}(Y\hash Y)$ of $\mathrm{Tor}^{\Rin}_{*,*}(\ztwo,\ztwo)$. The latter was described in detail in Section \ref{projres}. As $d_2$ has bidegree $(-2,1)$, we see that the only possible non-trivial $d_2$ differentials are from a $\ztwo$ summand in bidegree $(2n,-3n)$ to a $\ztwo$ summand in bidegree $(2n-2,-3n-1)$ for $n\geq 2$. On the other hand, the former is generated by $Q_{2n}=Q^2|Q|\cdots|Q^2|Q$, and we have that $d_2(Q_{2n})=0$ as it follows from Lemma \ref{d2} and the fact that
\begin{equation*}
\langle Q,Q^2,Q\rangle=\langle Q^2,Q,Q^2\rangle=0,
\end{equation*}
as it follows from Theorem \ref{Massey} (in this simple case, we also have simple direct proofs, see Example \ref{chainvan} for the first and \cite{Lin4} for the second). Regarding $d_3$, the natural generalization of Lemma \ref{d2} describes it in terms of four-fold Massey products; and the product $\langle Q,Q^2,Q,Q^2\rangle=V$ implies that we have the differentials
\begin{equation*}
d_3(Q_{i})=\mathrm{sh}(Q_{i-4},V).
\end{equation*}
for $i\geq 5$. Graphically, we see the differentials
\begin{center}
\begin{tikzpicture}
\matrix (m) [matrix of math nodes,row sep=0.1em,column sep=0.7em,minimum width=0.1em]
  { \underline{\ztwo}\\
  &\underline{\ztwo}\\
  &\cdot\\
  &\cdot&\underline{\ztwo}\\
  &\ztwo&\cdot&\underline{\ztwo}\\
  &&\ztwo&\cdot&\\
  &&&\cdot&\ztwo\\
  &&&\ztwo&\cdot&\ztwo\\
  &&&&\ztwo&\cdot\\
  &&&&&\cdot&\ztwo\\
  &&&&&\ztwo&\cdot&\ztwo\\
  &&&&&&\ztwo\\};
  
  \path[-stealth]

(m-7-5) edge[dashed, ] node {}(m-5-2)
(m-8-6) edge[dashed, ] node {}(m-6-3)
(m-10-7) edge[dashed, ] node {}(m-8-4)
(m-11-8) edge[dashed, ] node {}(m-9-5)

;

\end{tikzpicture}
\end{center}
repeating in a $2$-periodic fashion. In particular, the spectral sequence collapses at the $E^4$ page, and (as there is no space for extensions), the final $\Rin$-module is a copy of $\ztwo^2$ in degrees $0$ and $-1$, corresponding to the underlined copies of $\ztwo$.
\end{example}
\vspace{0.3cm}
\textit{Connected sum with $M_{2k}$. }We observe that the $\Rin$-module $M_{2k}$ has a very nice $2$-step graded minimal free resolution
\begin{equation*}
M_{2k}\stackrel{d_0}{\longleftarrow} \Rin\langle4k-1\rangle\oplus\Rin\stackrel{d_1}{\longleftarrow}\Rin\langle-1\rangle
\end{equation*}
where $d_0$ sends, in the notation introduced in Section \ref{examples},
\begin{align*}
(1,0)&\mapsto qv^{-k}\\
(0,1)&\mapsto 1
\end{align*}
and $d_1$ is given in matrix notation by the matrix $(V^k,Q)$. In particular, the $E^2$-page of the Eilenberg-Moore spectral sequence for the connected sum $Y\hash M_k$ is supported on the first two colums; this implies that there are no higher differentials, so that $E^{\infty}\cong E^2$, and all we need is to understand is the extension problem. Recall furthermore that
\begin{equation*}
\mathrm{Tor}_{0,*}(\HSf_{\bullet}(Y),M_k)=\HSf_{\bullet}(Y)\otimes_{\Rin}M_k,
\end{equation*}
and by the discussion in the previous section $\mathrm{Tor}_{1,*}(\HSf_{\bullet}(Y),M_k)$ is in bijection with elements $\x\in\mathrm{ker}V^k\cap \mathrm{ker}Q$ via the assignment
\begin{equation*}
\x\rightarrow \x|V^k|qv^{-k}+\x|Q|1.
\end{equation*}
Hence, we need to understand the action of $\Rin$ on such an element. We have the following.
\begin{prop}\label{Tor1}
We have the identity
\begin{equation*}
Q\cdot (\x|V^k|qv^{-k}+\x|Q|1)=\langle Q,\x,V^k\rangle|qv^{-k}\in E^{\infty}_{0,*}.
\end{equation*}
If $V\x\neq0$, we have
\begin{equation*}
V\cdot (\x|V^k|qv^{-k}+\x|Q|1)=V\x|V^k|qv^{-k}+V\x|Q|1\in E^{\infty}_{1,*}.
\end{equation*}
while if $V\x=0$ we have
\begin{equation*}
V\cdot (\x|V^k|qv^{-k}+\x|Q|1)=\langle V,\x,Q\rangle|1\in E^{\infty}_{0,*}.
\end{equation*}
\end{prop}

In particular, the $\Rin$-module structure of $Y\hash M_k$ is determined entirely by the triple Massey products of the form $\langle Q,\x,V^k\rangle$, which we have described in Section \ref{Massey} in terms of the Gysin exact triangle. Let us discuss a simple example (see also Section \ref{moreex} for more examples).
\begin{example}
Let us compute the homology of $-M_2\hash M_4$, using the result above. The $E^2$ page is computed to be, graphically,
\begin{center}
\begin{tikzpicture}
\matrix (m) [matrix of math nodes,row sep=0.1em,column sep=0.7em,minimum width=0.1em]
  {
  &&&&\ztwo&\ztwo&\cdot&\cdot&\ztwo&\ztwo&\cdot&\ztwo&\ztwo&\ztwo&\cdots   \\
  \ztwo&\cdot &\ztwo&\cdot&\cdot&\cdot&\ztwo\\
  &&&&&&&\ztwo&&\ztwo\\
  &&&&&&&&\ztwo\\};

\draw (-5,-0.6) -- (5,-0.6);

\path[-stealth]
(m-1-6) edge[bend right] node {}(m-2-7)
(m-2-3) edge[bend right] node {}(m-2-7)
(m-2-1) edge[bend right, dotted] node {}(m-2-3)
(m-4-9) edge[dashed] node {}(m-1-9)
;
\end{tikzpicture}
\end{center}
Above the line, we have depicted $\mathrm{Tor}_{0,*}^{\Rin}(-M_2,M_4)=-M_2\otimes_{\Rin}M_4$. The first row, which is generated over $\Rin$ by $v|qv^{-2}$ and $v|1$, consists of based elements. The first two elements in the second row are $\z|qv^{-2}$ and $q^2|qv^{-2}$; we have depicted with a dotted arrow the Massey product relating them (whose existence follows from Lemma \ref{-Mn}). The solid arrows represent the non obvious $\Rin$-actions. Under the line, we represented $\mathrm{Tor}_{1,*}^{\Rin}(-M_2,M_4)$, which, by the lemma above, corresponds to $\mathrm{ker}V^k\cap \mathrm{ker}Q=\{\z\}$. It is represented by the element $\z|V^2|qv^{-2}+\z|Q|1$, and its image under the action of $Q$ is
\begin{equation*}
Q\cdot(\z|V^2|qv^{-2}+\z|Q|1)=\langle Q,\z,V^2\rangle|qv^{-2}=v^2|qv^{-2}=v|qv^{-1},
\end{equation*}
as depicted by the dashed arrow. To sum up, the final result is
\begin{center}
\begin{tikzpicture}
\matrix (m) [matrix of math nodes,row sep=0.1em,column sep=0.7em,minimum width=0.1em]
  {
  &&&&\ztwo&\ztwo&\cdot&\ztwo&\ztwo&\ztwo&\cdot&\ztwo&\ztwo&\ztwo&\cdots   \\
  \ztwo&\cdot &\ztwo&\cdot&\cdot&\cdot&\ztwo&&\oplus\\
  &&&&&&&\ztwo&&\ztwo\\};

\path[-stealth]
(m-1-6) edge[bend right] node {}(m-2-7)
(m-2-3) edge[bend right] node {}(m-2-7)
(m-2-1) edge[bend right, dotted] node {}(m-2-3)
(m-3-8) edge[bend right, dotted] node {}(m-3-10)
;
\end{tikzpicture}
\end{center}
and in particular $\alpha=\beta=2$, $\gamma=0$.
\end{example}

In fact, we have the following more general observation.
\begin{cor}\label{onlyR}
The Manolescu correction terms of $Y\hash M_k$ are determined the $\Rin$-module structure of $\HSf_{\bullet}(-Y)$.
\end{cor}
\begin{proof}
We need to understand the elements $\x\in \HSf_{\bullet}(Y)$ for which $\y=\langle Q,\x,V^k\rangle$ is based. As of course $p_*(\x)=0$, we have $\x=j(\x')$ for $\x'\in \HSt_{\bullet}(Y)$. If both $Q\x'=V^k\x'=0$, then we would have by naturality of Massey products
\begin{equation*}
\y=\langle Q,\x,V^k\rangle=\langle Q,j_*(\x'),V^k\rangle=j_*(\langle Q,\x',V^k\rangle),
\end{equation*}
so $\y$ would not be based. So one of the two products is non-zero, and in fact, after changing the choice of $\x'$ mapping to $\x$, we see that $V^k\x'=0$ and $Q\x'$ belongs to the tower. Hence the existence of a Massey product to a based element implies the existence of such an element in $\HSt_{\bullet}(Y)=\HSf_{\bullet}(-Y)$; the converse follows in an analogous fashion, see also the proof of Lemma \ref{-Mn}.
\end{proof}
In fact, the proof of the corollary implies that one can in principle write a (not particularly illuminating) formula for the  correction terms of $Y\hash M_k$ in terms of the $\Rin$-module structure of $\HSf(-Y)$
\\
\par
The key computation behind Proposition \ref{Tor1} are the following two general lemmas.
\begin{lemma}\label{ract}
Suppose $\x r=\x s=t\x=0$ and $r\y+s\z=0$, and consider the element $\x |r|\y+\x |s|\z\in \Tor^{\Rin}_{1,*}$. If it survives in the $E^{\infty}$-page, multiplication by $t$ sends it to $\langle t,\x ,r\rangle|\y+\langle r,\x ,s\rangle|\z$.
\end{lemma}
\begin{proof}
Suppose we have fixed cycles representing the homology classes $\x,\y,\z, r ,s$, which we denote with the same letter. Choose chains $a,b,c$ such that
\begin{equation*}
\partial a= \x r,\quad \partial b=\x s,\quad \partial c=r\y+s\z,\quad \partial d=t\x.
\end{equation*}
Then
\begin{equation*}
\x |r|\y+\x |s|\z+a|\y+b|\z+\x|c
\end{equation*}
is a cycle in the $\Ainf$-tensor product whose image in the $E^2$-page is the class $\x |r|\y+\x |s|\z$. By definition, the action of $t$ on it is given by the cycle
\begin{equation*}
t\x |r|\y+t\x |s|\z+ta|\y+tb|\z+t\x|c+m_3(t|\x|r)|\y+m_3(t|\x|s)|\z.
\end{equation*}
Now, we have the identities
\begin{align*}
\partial(d|r|\y+d|s|\z)&=t\x |r|\y+t\x |s|\z+dr|\y+d|r\y+ds|\z+d|s\z\\
\partial (d|c)&=t\x|c+d|r\y+d|s\z
\end{align*}
so that summing all the three equations we see that a representative of the action by $t$ is
\begin{equation*}
(ta+dr+m_3(r|\x|r))|\y+(tb+ds+m_3(r|\x|s))|\z,
\end{equation*}
hence the result.
\end{proof}

The second lemma is the following.

\begin{lemma}
Suppose that $r\x=0$. Then $\langle r,\x,r\rangle=0$.
\end{lemma}

The proof of this lemma is a special case of a more general result involving generalized Massey products that arise when studying higher differentials of elements which are not represented by simple tensors (see Lemma \ref{d2} and the following discussion). The simplest case is the following. Let $\x$ and consider two classes $r,s$ with $\x r=\x s=0$, where we do \textit{not} assume $rs=0$. Choose chains $a,b,c$ such that
\begin{equation*}
\partial a=\x r,\quad \partial b= \x s,\quad \partial c=rs+sr.
\end{equation*}
Then the following expression
\begin{equation}\label{compoundchain}
as+br+\x c+m_3(\x|s|r)+m_{2,2}(\x|r|s)
\end{equation}
is a cycle, and we define its homology class to be the \textit{generalized Massey product} $\langle\x |r,s\rangle$. We then have the following.
\begin{lemma}\label{compound}
Given classes $\x,r$ and $s$ as above, we have $\langle\x |r,s\rangle=\langle r,\x,s\rangle$.
\end{lemma}

For example, for $-M_2$ the identity
\begin{equation*}
\langle \z|Q,V\rangle=\langle Q,\x,V\rangle=q^2
\end{equation*}
holds, see Lemma \ref{-Mn}.

\begin{proof}
Let us first spell out some details implicit in the construction of the $\Ainf$-bimodule structure in \cite{Lin4}. Suppose the family of metrics and perturbations has been chosen so that the $\Ainf$-module structure $\{m_n\}$ is defined (this implicitly takes as input an embedded ball in $Y$). The $\Ainf$-bimodule structure takes implicitly as input a second embedded ball, disjoint from the first. Then, by suitably pulling back (as in the construction of \cite{Lin4}) the metric used to define $\{m_n\}$ to our new data (defined by the second embedded ball), we see that we can assume that our multiplication safisfies $\mu_{2,1}(d|\y)=\mu_{1,2}(\y|d)$ for all choices of $d\in \hat{C}^{\jmath}_{\bullet}(S^3)$ and  $\y\in \hat{C}^{\jmath}_{\bullet}(Y)$. Notice that this does not imply that the multiplication $\mu_2$ on $\hat{C}^{\jmath}_{\bullet}(S^3)$ is commutative at the chain level (rather than just commutative up to homotopy), as the data on $S^3$ has been fixed a priori. 
The implication for our purposes is that, once this choice of data is made, we have that $\langle r,\x,s\rangle$
is represented by
\begin{equation*}
as+rb+m_3(r|\x|s),
\end{equation*}
where we are using the notation introduced above. To show that this cycle is cobordant to the one in Equation (\ref{compoundchain}), let us consider a family of metrics and perturbations on the manifold with cylindrical ends
\begin{equation*}
\left(I\times Y\setminus(  \mathrm{int}D^4\amalg\mathrm{int}D^4)\right)^*
\end{equation*}
parametrized by a hexagon as in Figure \ref{boundary}, and consider the chain obtained by taking fibered products on the incoming end with $\x, r$ and $s$. We provide a sketchy description of the metrics and perturbations involved - the details of the construction are very similar to those in \cite{Lin4}. The thick edges of the hexagon $\mathcal{H}$ correspond to stretching along the three pairs of hypersurfaces on the right of the figure, and taking fibered products one obtains the chains $m_3(\x|r|s), m_3(\x|s|r)$ and $m_{2,2}(r|\x|s)$ respectively. The top thin edge corresponds to a metric in which the top hypersurface is streched to infinity, and we perform a chain homotopy realizing the commutativity of $\mu_2$ on $\hat{C}^{\jmath}_{\bullet}(S^3)$; the corresponding chain is $c$. The bottom thin lines correspond to a chain homotopy between $m_{2,1}$ and $m_{1,2}$ with one of the diagonal hypersurfaces stretched to infinity; as discussed above we can choose such a chain homotopy to be induced by a isotopy of metrics, so that the chains in consideration will just be $I\times (\x r)s$ and $I\times (\x s)r$, hence they are zero in our chain complex (as they are small, see Section \ref{review}). The boundary of the hexagon can be filled with a family of metrics and perturbations of the manifold with cylindrical ends, as the corresponding space is contractible; taking the fibered product with the moduli spaces parametrized by $\mathcal{H}$, we obtain a chain whose boundary (from the discussion above) is $m_3(\x|r|s)+m_3(\x|s|r)+m_{2,2}(r|\x|s)+\x c$ and the result follows.
\end{proof}
\begin{figure}
  \centering
\def\svgwidth{0.9\textwidth}
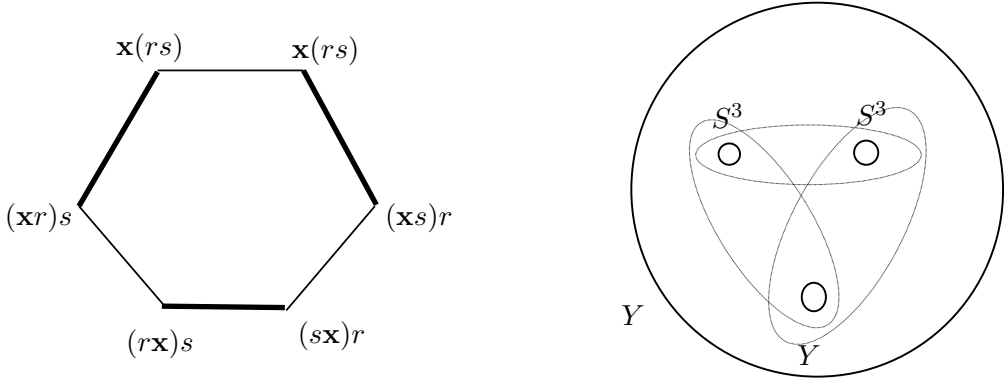
\caption{On the left, the hexagon $\mathcal{H}$ parametrizing the family of metrics and perturbations, where we have denoted the strata corresponding to each corner. On the right, we have depicted three hypersurfaces in a doubly punctured $I\times Y$; the top one is a copy of $S^3$ while the diagonal ones are copy of $Y$. Each pair determines a one parameter family of metrics and perturbations defining higher compositions $m_3$ or $m_{2,2}$.}
\label{boundary}
\end{figure} 

\begin{proof}[Proof of Proposition \ref{Tor1}]
Given the lemmas above, Proposition \ref{Tor1} follows immediately, with the additional observation for the last point that when $V\x=0$, $\x|V^k|qv^{-k}+\x|Q|1$ is also represented by $\x|V|qv^{-1}+\x|Q|1$, as they differ by the boundary of $\x|V|V^{k-1}|qv^{-k}$.
\end{proof}
\vspace{0.3cm}

\textit{Connected sums with $-M_{2k}$. }We now discuss the effect of connected sums with manifolds of type $-M_{2k}$, i.e. manifolds obtained by manifolds of type $M_{2k}$ by reversing the orientation (see Section \ref{examples}). Of course, the correction terms of $Y\hash-M_{2k}$ are determined by the correction terms of $-(Y\hash-M_{2k})=-Y\hash M_{2k}$, so that Corollary \ref{onlyR} implies in turn that they are determined entirely by the module structure of $\HSf_{\bullet}(Y)$. We will explain this fact by a direct inspection of the Eilenberg-Moore spectral sequence.
\\
\par
For simplicity, we will start with the case $k=1$, in which case $-M_2=\ztwo\oplus N_2$ as $\Rin$-modules. A convenient projective resolution for the trivial module $\ztwo$ was defined in Section \ref{projres}. The module $N$ also admits a simple $2$-periodic projective resolution
\begin{equation*}
N_2\stackrel{d_0}{\longleftarrow}\Rin\oplus \Rin\stackrel{d_1}{\longleftarrow} \Rin\oplus\Rin\stackrel{d_2}{\longleftarrow} \Rin\oplus\Rin\stackrel{d_3}{\longleftarrow} \Rin\oplus\Rin\stackrel{d_4}{\longleftarrow}\cdots
\end{equation*}
where, in matrix notation, and for $i\geq 0$ we have
\begin{equation*}
d_i=
\begin{cases}
\begin{bmatrix}
Q & 0\\
V & Q^2
\end{bmatrix}& \text{ if }i\text{ is even}\\
\begin{bmatrix}
Q^2 & 0\\
V & Q
\end{bmatrix}& \text{ if }i\text{ is odd}
\end{cases}
\end{equation*}
Here $d_0$ sends $(1,0)$ to $v$ and $(0,1)$ to $q^2$.
\par
As usual $\mathrm{Tor}^{\Rin}_{0,\ast}(M,N_2)=M\otimes_{\Rin}N_2$. Then the general description of generators of $\mathrm{Tor}^{\Rin}_{i,\ast}$, as given in Section \ref{projres}, specializes to our case as follows. Define
\begin{equation*}
\tilde{Q}_n=\begin{cases}
(Q^2,Q,\cdots, Q,Q^2)&\text{ if }n\text{ is odd}\\
(Q,Q^2,\cdots, Q,Q^2)&\text{ if }n\text{ is even}.
\end{cases}
\end{equation*}
\begin{lemma}\label{descrN2}
Let $n\geq 1$. Then every element of $\mathrm{Tor}^{\Rin}_{n,\ast}(M,N_2)$ has a representative of the form
\begin{equation*}
\x\lvert \tilde{Q}_n \lvert v+\x\lvert \mathrm{sh}(Q_{n-1},V)\lvert q^2+ \y\lvert Q_n\lvert q^2
\end{equation*}
where $d_{n}(\x,\y)=0$.
\end{lemma}

The elements in $\mathrm{Tor}^{\Rin}_{*,*}(M,\ztwo)$ were described in detail in Example \ref{ztwotor}. Consider first an element in $\mathrm{Tor}^{\Rin}_{2,*}$, which has the form
\begin{equation*}
\alpha=\A \lvert Q\lvert V\lvert \z+\A \lvert V\lvert Q\lvert \z+\B \lvert Q^2\lvert Q\lvert \z
\end{equation*}
where $Q\A=0$ and $V\A+Q^2\B=0$. Such an element is zero if and only if $\A=Q^2\x$ and $\B=V\x+Q\y$. Give elements $\A,\B$ satisfying $Q\A=0$ and $V\A+Q^2\B=0$ we can define their generalized triple Massey product $\phi(\A,\B)$ as follows in the same way as Equation (\ref{compoundchain}). Then the differential on the $E^2$-page is identified with
\begin{align*}
d_2(\alpha)&=\phi(\A,\B)\lvert \z +\A\lvert (\langle Q, V, \z\rangle+\langle V, Q, \z\rangle)+\B\lvert \langle  Q^2, Q, \z\rangle\\
&=\phi(\A,\B)\lvert \z + \A\lvert v+\B\lvert q^2.
\end{align*}
where in the last row we used Lemmas \ref{-Mn} and \ref{compound}. This readily generalizes to the case of elements in $\mathrm{Tor}^{\Rin}_{n,*}$ with $n$ even. Specializing the discussion of Section \ref{projres}, a general element is of the form
\begin{equation*}
\alpha=\A\lvert \mathrm{sh}(Q_{n-1},V)\lvert \z+ \B\lvert Q_n\lvert \z
\end{equation*}
where $Q\A=0$ and $V\A+Q^2\B=0$, and such an element is zero if and only if $\A=Q^2\x$ and $\B=V\x+Q\y$. Its differential is then given by
\begin{align*}
d_2(\alpha)&=\langle \A,Q,Q^2\rangle\lvert \mathrm{sh}(Q_{n-3},V)\lvert \z+\phi(\A,\B)\lvert Q_{n-2}\lvert \z+\\
&+a\lvert \tilde{Q}_{n-2}\lvert v+\A\lvert \mathrm{sh}(Q_{n-3},V)\lvert q^2+ \B\lvert Q_{n-2}\lvert q^2 \numberthis \label{d2a}
\end{align*}
The key observation to have in mind is that this differential naturally decomposes in two parts: the one in the first row, which defines an element in $\mathrm{Tor}^{\Rin}_{n-2,*}(M,\ztwo)$, and the one in the second row in $\mathrm{Tor}^{\Rin}_{n-2,*}(M,N_2)$. In fact, the latter is the element corresponding to the pair $(\A,\B)$ from Lemma \ref{descrN2}.
\\
\par
An analogous description holds to the odd case too (provided of course that $n\geq 3$). For the sake of clarity, we first write down explicitly the case of $\mathrm{Tor}^{\Rin}_{3,*}(M,\ztwo)$. An element in this group has the form
\begin{equation*}
\beta=\x \lvert Q^2 \lvert Q\lvert V\lvert \z+\x \lvert Q^2\lvert V\lvert Q\lvert \z+\x\lvert V\lvert Q^2\lvert Q\lvert \z+\y\lvert Q \lvert Q^2\lvert Q\lvert \z
\end{equation*}
where $Q^2\x=0$ and $V\x+Q\y=0$. Such an element is zero if and only if $\x=Q\A$ and $\y=V\A+Q^2 \B$. To a pair $\x,\y$ like this we can assign a generalized triple Massey product as in Equation (\ref{compoundchain}). We then have (using again the computation in Lemma \ref{-Mn})
\begin{equation*}
d_2(\beta)=\langle \x,Q^2,Q\rangle\lvert V\lvert \z+\psi(\x,\y)\lvert Q\lvert \z+\x\lvert Q^2\lvert v+\x\lvert V\lvert q^2+\y\lvert Q\lvert q^2.
\end{equation*}
In the general case for $n$ odd, we have that the general element of $\mathrm{Tor}^{\Rin}_{n,*}(M,\ztwo)$ is
\begin{equation*}
\beta=\x\lvert \mathrm{sh}(Q_{n-1},V)\lvert \z+ \y\lvert Q_n\lvert \z,
\end{equation*}
where $\x$ and $\y$ are as above. We then have
\begin{align*}
d_2(\beta)&=\langle \x,Q^2,Q\rangle\lvert \mathrm{sh}(Q_{n-3},V)\lvert \z+\psi(\x,\y)\lvert Q_{n-2}\lvert \z+\\
&+\x\lvert \tilde{Q}_{n-2}\lvert v+\x\lvert \mathrm{sh}(Q_{n-3},V)\lvert q^2+\y\lvert Q_{n-2}\lvert q^2.\numberthis \label{d2b}
\end{align*}
Again, the first row is an element in $\mathrm{Tor}_{n-2,*}^{\Rin}(M,\ztwo)$, and the second row is the element in $\mathrm{Tor}_{n-2,*}^{\Rin}(M,N_2)$ corresponding to the pair $(\x,\y)$.
\\
\par
With this discussion in hand, we are ready to prove the following.
\begin{prop}\label{E3coll}
The Eilenberg-Moore spectral sequence for the connected sum with $-M_2$ collapses at the $E^3$-page. Furthermore, $E^{\infty}_{i,*}=0$ for $i\geq 2$, and all elements in $E^{\infty}_{1,*}$ have the form $\x|V|\z+\y|Q|\z$ for pairs $\x,\y$ of elements in $M$ such that $V\x+Q\y=0$. The action of $V$ and $Q$ on such an element are given respectively by $\x|v$ and $\y|q^2$.
\end{prop}
From this, it is clear as mentioned in the introduction that the correction terms of $Y\hash-M_2$ are determined entirely by the module structure of $Y$.
\begin{proof}
Set $M=\HSf_{\bullet}(Y)$. Let us consider first the part in degree $2n$ of the $E^2$-page, i.e.
\begin{equation*}
\mathrm{Tor}^{\Rin}_{2n,*}(M,-M_2)=\mathrm{Tor}^{\Rin}_{2n,*}(M,N_2)\oplus \mathrm{Tor}^{\Rin}_{2n,*}(\ztwo,\ztwo).
\end{equation*}
For $n\geq1$, both summands can be identified with
\begin{equation*}
V=\mathrm{ker}\begin{bmatrix}
Q^2 & 0\\
V & Q
\end{bmatrix}/\mathrm{im}\begin{bmatrix}
Q & 0\\
V & Q^2
\end{bmatrix}
\end{equation*} 
where we think the matrices as acting on $M^2$. Because of two periodicity, the part in even grading of the $(E^2,d_2)$ page can be rewritten as the complex
\begin{equation*}
\cdots \stackrel{d_2}{\longrightarrow} V\oplus V\stackrel{d_2}{\longrightarrow} V\oplus V\stackrel{d_2}{\longrightarrow} V\oplus V\stackrel{d_2}{\longrightarrow}\cdots \stackrel{d_2}{\longrightarrow}V\oplus V\stackrel{d_2}{\longrightarrow} V\oplus V.
\end{equation*}
When thought as a $2$-by-$2$ matrix,$d_2$ is upper triangular, and by (\ref{d2a}) above we see that has to be of the form
\begin{equation*}
d_2=\begin{bmatrix}
A & \mathrm{Id}_V\\
0 & B
\end{bmatrix}.
\end{equation*}
Imposing $d_2^2=0$ we also obtain $A^2=0$ and $A=B$, so that
\begin{equation*}
d_2=\begin{bmatrix}
A & \mathrm{Id}_V\\
0 & A
\end{bmatrix}.
\end{equation*}
We want to show that $d_2$ is exact. Suppose $d_2(\x,\y)=0$. Hence, $A\x=\y$ and $A\y=0$. Then
\begin{equation*}
d_2(0,\x)=\begin{bmatrix}
A & \mathrm{Id}_V\\
0 & A
\end{bmatrix}
\begin{bmatrix}
0\\
\x
\end{bmatrix}=\begin{bmatrix}
\x\\
\y
\end{bmatrix}
\end{equation*}
and the result follows.
\par
The analogous argument holds for the odd part, as for each $n\geq1$
\begin{equation*}
\mathrm{Tor}^{\Rin}_{2n+1,*}(M,-M_2)=\mathrm{Tor}^{\Rin}_{2n+1,*}(M,N_2)\oplus \mathrm{Tor}^{\Rin}_{2n+1,*}(\ztwo,\ztwo).
\end{equation*}
each of the summands on the right can be identified with
\begin{equation*}
W=\mathrm{ker}\begin{bmatrix}
Q & 0\\
V & Q^2
\end{bmatrix}/\mathrm{im}\begin{bmatrix}
Q^2 & 0\\
V & Q
\end{bmatrix}
\end{equation*} 
and $d_2$ has an analogous shape (see equation (\ref{d2b}). Finally, the $\Rin$-module structure can be easily computed using Lemma \ref{-Mn} and \ref{ract}. 
\end{proof}

The general case of $M_{2k}$ can be derived with few modifications, so that in particular the analogue of Proposition \ref{E3coll} also holds. Let us discuss the key points involved. Recall that as $\Rin$-modules, $M_{2k}=N_{2k}\oplus \ztwo[V]/V^k$. The module $\ztwo[V]/V^k$ has minimal free resolution
\begin{equation*}
\ztwo[V]/V^k\stackrel{d_0}{\longleftarrow} \Rin\stackrel{d_1}{\longleftarrow} \Rin\oplus\Rin\stackrel{d_2}{\longleftarrow} \Rin\oplus\Rin\stackrel{d_3}{\longleftarrow} \Rin\oplus\Rin\stackrel{d_4}{\longleftarrow}\cdots
\end{equation*}
where, in matrix notation, $d_1=(V^k,Q)$, and for $i\geq 2$ we have
\begin{equation*}
d_i=
\begin{cases}
\begin{bmatrix}
Q & 0\\
V^{k} & Q^2
\end{bmatrix}& \text{ if }i\text{ is even}\\
\begin{bmatrix}
Q^2 & 0\\
V^{k} & Q
\end{bmatrix}& \text{ if }i\text{ is odd}
\end{cases}
\end{equation*}
while $N_{2k}$ has minimal projective resolution
\begin{equation*}
N_{2k}\stackrel{d_0}{\longleftarrow}\Rin\oplus \Rin\stackrel{d_1}{\longleftarrow} \Rin\oplus\Rin\stackrel{d_2}{\longleftarrow} \Rin\oplus\Rin\stackrel{d_3}{\longleftarrow} \Rin\oplus\Rin\stackrel{d_4}{\longleftarrow}\cdots
\end{equation*}
where, in matrix notation, and for $i\geq 0$ we have
\begin{equation*}
d_i=
\begin{cases}
\begin{bmatrix}
Q & 0\\
V^{n+1} & Q^2
\end{bmatrix}& \text{ if }i\text{ is even}\\
\begin{bmatrix}
Q^2 & 0\\
V^{n+1} & Q
\end{bmatrix}& \text{ if }i\text{ is odd}.
\end{cases}
\end{equation*}
Here $d_0$ sends $(1,0)$ to $v^n$ and $(0,1)$ to $q^2$. The result then follows from the identities from Section \ref{examples}, $\langle V^{n},\z,Q\rangle=v^n$ and $\langle \z,Q,Q^2\rangle=q^2$, in the same way as for $-M_2$.
\\
\par
\textit{The odd case.} We have seen in the previous section that connected sums with manifolds of type $\pm M_{2k}$ are rather simple to understand. We will discuss the case in which $n=2k+1$ is odd. Let us first discuss a suitable minimal projective resolution.
\begin{equation*}
M_k\stackrel{\epsilon}{\longleftarrow}\Rin\oplus \Rin\stackrel{d_0}{\longleftarrow} \Rin\oplus\Rin\stackrel{d_1}{\longleftarrow} \Rin\oplus\Rin\stackrel{d_2}{\longleftarrow} \Rin\oplus\Rin\stackrel{d_3}{\longleftarrow}\cdots
\end{equation*}
where
\begin{equation*}
d_i=
\begin{cases}
\begin{bmatrix}
Q & 0\\
V^{n+1} & Q^2V^n
\end{bmatrix}& \text{ if }i=0\\
\begin{bmatrix}
Q & 0\\
V & Q^2
\end{bmatrix}& \text{ if }i\text{ is even}\\
\begin{bmatrix}
Q^2 & 0\\
V & Q
\end{bmatrix}& \text{ if }i\text{ is even}
\end{cases}
\end{equation*}
and $d_0$ maps $(1,0)$ to $qv^{n+1}$ and $(0,1)$ to $1$. The computation of the $E^2$ page, from the general description of Section \ref{projres}, is the following.
\begin{lemma}\label{torodd}
Let $n$ be odd. Then the generators of $\mathrm{Tor}_{n,*}(M,N_{1})$ have representatives of the form
\begin{equation*}
\x|Q_n|v+\mathrm{sh}(\tilde{Q}_{n-1},V)|q+\y|\tilde{Q}_n|q.
\end{equation*}
for each $Q\x=V\x+Q^2\y=0$. Let $n$ be even. Then the generators of $\mathrm{Tor}_{n,*}(M,N_1)$ have representatives of the form
\begin{equation*}
\x|\tilde{Q}_n|v+\x|\mathrm{sh}(Q_{n-1},V)|q+\y|Q_n|q.
\end{equation*}
The description in the case of $N_{2k+1}$ with $k\geq 1$ is analogous.
\end{lemma}
The description of the $E^{\infty}$ page is not as straightforward as in the even case. While the differential on $E^2$ can be described as in the even case in terms of certain generalized Massey products, the spectral sequence does non collapse at the $E^3$-page in general. On the other hand, because the relation
\begin{equation*}
\langle QV^i,Q^2V^j,QV^,Q^2V^\rangle=V^{1+i+j+k+l}
\end{equation*}
holds, it can be shown that the spectral sequence collapses at the $E^4$ page (see Example \ref{FF}). Rather than discussing the quite involved general theory, let us work out in detail a specific example that enlightens the key aspects of the computation.

\begin{example}\label{M1M1}
Let us revisit the example of the connected sum $M_1\hash M_1$ discussed in \cite{Lin4} from our new perspective. The $E^2$ page of the Eilenberg-Moore spectral sequence is depicted below. Here, starting from the left the $i$th column represents $\mathrm{Tor}^{\Rin}_{i,*}$, with the element on the top left having bidegree $(0,2)$. The picture repeats $2$-periodically to the right as in Example \ref{FF}.
\begin{center}
\begin{tikzpicture}
\matrix (m) [matrix of math nodes,row sep=0.1em,column sep=0.9em,minimum width=0.1em]
{\ztwo \\
\cdot&\ztwo\\
\cdot&&\underline{\ztwo}\\
\ztwo&\oplus\ztwo\\
\ztwo&&&\underline{\ztwo}\\
\cdot&&\ztwo&&\ztwo\\
\ztwo&&&\ztwo\\
\ztwo&&&&&\ztwo\\
\ztwo&&&&\ztwo&&\ztwo\\
\cdot&&&&&\ztwo\\
};
\path[-stealth]
(m-1-1) edge[bend left] node {}(m-2-2)
(m-6-5) edge[dotted] node {}(m-4-2)
(m-8-6) edge[dotted] node {}(m-6-3)
(m-9-7) edge[dotted] node {}(m-7-4)
;
\draw  (-1,-3) -- (-1,3);
\draw[dotted]  (-3,1.8) -- (3,1.8);

\end{tikzpicture}
\end{center}
The first column in $\mathrm{Tor}^{\Rin}_{0,*}$ is generated by the based element
\begin{equation*}
q|q,\quad q|v,\quad v|v,
\end{equation*}
while the additional summand $\oplus\ztwo$ is generated by $q|v+v|q$. The elements of the higher $\mathrm{Tor}^{\Rin}_{i,*}$ groups can be described thanks to Lemma \ref{torodd}. In particular, the generators of the top summand of each of the first three columns are given by
\begin{equation*}
q|Q^2|q,\quad Qq|Q|Q^2|q,\quad q|Q^2|Q|Q^2|q.
\end{equation*}
The differential $d_2$ (described in Lemma \ref{d2}) vanishes thanks to our description of the Massey products $\langle\cdot,Q^2,Q\rangle$ and $\langle\cdot, Q,Q^2\rangle$ in Theorem \ref{Massey}. On the other hand, the differential $d_3$ is non-trivial. In the picture, the top dotted arrows represents
\begin{equation*}
d_3(q|Q^2|Q|Q^2|q)=\langle q,Q^2,Q,Q^2\rangle|q+q|\langle Q^2,Q,Q^2,q\rangle=v|q+q|v
\end{equation*}
where the Massey product $\langle q,Q^2,Q,Q^2\rangle=v$ is computed as in Remark \ref{kad}. Similarly, the other two dotted arrows represent
\begin{align*}
d_3(Qq|Q|Q^2|Q|Q^2|q)&=v|Q^2|q+Qq|V|q+Qq|Q|v\\
d_3(q|Q^2|Q|Q^2|Q|Q^2|q)&=v|Q|Q^2|q+q|V|Q^2|q+q|Q^2|V|q+q|Q^2|Q|v,\\
\end{align*}
and in general the whole $2$-periodic tail cancels out in this fashion as in Example \ref{FF}. In particular, of the $\mathrm{Tor}^{\Rin}_{i,*}$ for $i\geq1$ only the two underlined summands survive to the $E^{\infty}$ page. There is only one non trivial extension, namely
\begin{equation*}
Q\cdot(Qq|Q|Q^2|q)=v|q,
\end{equation*}
which can be again computed thanks to Theorem \ref{Massey}. The final result is therefore graphically depicted as
\begin{center}
\begin{tikzpicture}
\matrix (m) [matrix of math nodes,row sep=0.1em,column sep=0.7em,minimum width=0.1em]
  { \ztwo&\cdot &\ztwo&\ztwo&\ztwo&\cdots  \\
  &\ztwo\\
  &\ztwo\\};

\path[-stealth]

(m-1-1) edge[bend right] node {}(m-2-2)
(m-1-4) edge[bend right] node {}(m-1-5)
(m-1-3) edge[bend right] node {}(m-1-4)
(m-1-1) edge[bend left] node {}(m-1-5)

;
\end{tikzpicture}
\end{center}
so that $\alpha=2$ and $\beta=\gamma=0$.
\end{example}

\vspace{0.3cm}

\begin{example}\label{M3M3}
Consider the connected sum of two manifolds of simple type $M_3$. The computation of the $E^2$ page is showed below. The group $\mathrm{Tor}^{\Rin}_{0,*}$ is as usual the tensor product $M_3\otimes_{\Rin}M_3$; the first column consists of based elements and is generated over $\Rin$ by
\begin{equation*}
qv^{-1}|qv^{-1},\quad qv^{-1}|v,\quad v|v
\end{equation*}
lying in degrees respectively $-10$, $-3$ and $4$, while the summands $\oplus\ztwo$ are unbased and  are represented by $qv^{-1}|v+v|qv^{-1}$ and its image under $V$ respectively.
Again there is a $2$-periodic infinite tail as in the case of $\mathrm{Tor}^{\Rin}_{*,*}(\ztwo,\ztwo)$; its top generator in each of the first three columns is given respectively by
\begin{equation*}
q|Q^2V|qv^{-1},\quad Qq|Q|Q^2V|qv^{-1},\quad q|Q^2|Q|Q^2V|qv^{-1}.
\end{equation*}
The main difference in this case is the presence of the extra summand $F$, which is represented by $qv^{-1}|Q^2V|qv^{-1}$, and corresponds to
\begin{equation*}
(0,qv^{-1})\in\mathrm{ker}\begin{bmatrix}
Q & 0\\
V^{2} & Q^2V
\end{bmatrix}.
\end{equation*}
\begin{center}
\begin{tikzpicture}
\matrix (m) [matrix of math nodes,row sep=0.1em,column sep=0.9em,minimum width=0.1em]
  {\ztwo   \\
 \cdot &\ztwo\\
  \cdot&\ztwo\\
  \cdot\\
  \ztwo\\
  \cdot&{\ztwo}\\
  \cdot &&&\underline{F}\\
  \ztwo&\oplus\ztwo&\\
  \ztwo\\
  \cdot&{\ztwo}\\
  \cdot&&&\underline{\ztwo}\\
  \ztwo&\oplus\ztwo&\\
  \ztwo&&&&\underline{\ztwo}\\
  \cdot&&&\ztwo&&\ztwo\\
  \ztwo&&&&\ztwo\\
  \ztwo&&&&&&\ztwo\\
  \ztwo&&&&&\ztwo&&\ztwo\\
  \cdot&&&&&&\ztwo\\
  };

\path[-stealth]
(m-1-1) edge[bend left] node {}(m-2-2)
(m-2-2) edge[bend left] node {}(m-3-2)
(m-5-1) edge[bend left] node {}(m-6-2)
(m-8-1) edge[bend left] node {}(m-9-1)
(m-9-1) edge[bend left] node {}(m-10-2)
(m-12-1) edge[bend left] node {}(m-13-1)
(m-15-1) edge[bend left] node {}(m-16-1)
(m-16-1) edge[bend left] node {}(m-17-1)
(m-16-7) edge[bend right, dotted] node {}(m-14-4)
(m-17-8) edge[bend right, dotted] node {}(m-15-5)
(m-14-6) edge[bend right, dotted] node {}(m-12-2)
(m-7-4) edge[bend left] node {}(m-11-4)
(m-8-2) edge[bend left] node {}(m-12-2)
;
\draw  (-1,-6) -- (-1,6);
\draw  (-0.25,-6) -- (-0.25,6);
\draw  (0.75,-6) -- (0.75,0);
\draw  (1.5,-6) -- (1.5,0);
\draw  (2.5,-6) -- (2.5,0);
\draw[dotted](-3.5,-1.1) -- (3,-1.1);

\end{tikzpicture}
\end{center}

From this description, we readily recover the $E^{\infty}$-page as in Example \ref{M1M1}. In particular $E^{\infty}_{i,*}$ vanishes for $i\geq 3$, and of $\mathrm{Tor}^{\Rin}_{i,*}$, $i\geq1$, only the underlined summands survive. The only non-trivial extension is given as in Example \ref{M1M1} by
\begin{equation*}
Q\cdot(Qq|Q|Q^2V|qv^{-1})=v^2|qv^{-1}=v|q\\
\end{equation*}
The Floer homology of $M_3\hash M_3$ then looks like the following:
\begin{center}
\begin{tikzpicture}
\matrix (m) [matrix of math nodes,row sep=0.1em,column sep=1em,minimum width=0.1em]
  {     \ztwo &\cdot&\cdot&  \cdot  & \ztwo & \cdot &\cdot & \ztwo &\ztwo&\cdot&\ztwo_0&\ztwo&\ztwo&\cdots \\
          & \ztwo &\ztwo   &&  & \ztwo &  &  & &\ztwo\\
          &&&&\ztwo&&\ztwo&&\ztwo\\
};
\path[-stealth]
(m-1-1) edge[bend left] node {}(m-1-5)
(m-1-5) edge[bend left] node {}(m-1-9)
(m-1-8) edge[bend left] node {}(m-1-12)
(m-1-9) edge[bend left] node {}(m-1-13)
(m-1-1) edge[bend right] node {}(m-2-2)
(m-2-2) edge[bend right] node {}(m-2-3)
(m-1-5) edge[bend right] node {}(m-2-6)
(m-1-8) edge[bend right] node {}(m-1-9)
(m-1-9) edge[bend right] node {}(m-2-10)
(m-1-11) edge[bend right] node {}(m-1-12)
(m-1-12) edge[bend right] node {}(m-1-13)
(m-3-5) edge[bend right] node {}(m-3-9)
;
\end{tikzpicture}
\end{center}
In particular, $\alpha=6$, $\beta=2$ and $\gamma=0$.\end{example}
\begin{remark}
Recall that for the usual monopole Floer homology we have that
\begin{equation*}
\HMf_{\bullet}(Y_0\hash Y_1,\spin_0\hash \spin_1)=\mathrm{Tor}^{\ztwo[U]}_{*,*}(\HMf_{\bullet}(Y_0,\spin_0),\HMf_{\bullet}(Y_1,\spin_1))\langle1\rangle,
\end{equation*}
see \cite{Lin4}. In particular, The usual monopole Floer homology of $M_3\hash M_3$ is given by
\begin{equation*}
\big(\ztwo[U]\oplus\ztwo[U]/U^3\langle5\rangle\oplus\ztwo[U]/U^3\langle10\rangle\big)\oplus (\ztwo[U]/U^3\langle5\rangle)^{\oplus2}.
\end{equation*}
The first summand is related via the Gysin exact triangle to the first two rows of our final result, while the second summand to the third row. This last computation implies the existence of some Massey products relating the three $\ztwo$ summands in the third row of $\HSf_{\bullet}$.
\end{remark}
\vspace{0.3cm}

\vspace{0.3cm}
\section{Connected sums with more summands.}\label{moreex}
In this final section we discuss more examples of connected sums (many of which already appeared in different setups, see \cite{Sto2}, \cite{DaiM}, \cite{DaiS}), involving manifolds of simple type (possibly with both orientations).
\\
\par
\textit{Sums of manifolds of simple type. }The following is the analogous of the main result of \cite{Sto2}.
\begin{prop}
Consider sequences of integers $0<n_1<n_2<\dots$, and suppose that, for each $i$, $Y_i$ has simple type $M_{n_i}$. Then the $Y_i$ are linearly independent in $\Th$.
\end{prop}
This holds for example for $Y_i=-\Sigma(2,4n-1,8n-1)$ (see \cite{FurCob} for the original proof, using instantons, that $\mathbb{Z}^{\infty}\subset\Th$).
\begin{proof}
The proof is essentially that of \cite{Sto2}. Let us focus for simplicity in the case all summands have even type. The crucial observation to show linear independence is the following: if we have a relation
\begin{equation*}
m_1[Y_{i_1}]+\dots+m_e[Y_{i_e}]=n_1[Y_{j_1}]+\dots+n_f[Y_{j_f}]
\end{equation*}
for positive $m_h$ and $n_h$, and $i_1<\cdots < i_e$,  $j_1<\cdots < j_f$, then $i_e=j_f$. This readily implies a unique factorization property. To show that is true, we claim that
\begin{equation*}
\beta(m_1[Y_{i_1}]+\dots+m_e[Y_{i_e}])=-2i_e.
\end{equation*}
so that the highest index among the summands is determined by the $\beta$ invariant of the sum. This can be computed inductively in the number of summands, using the description of connected sums with manifolds of simple type discussed in Section \ref{connsimple}. Let us for example consider the Floer homology of a connected sum $M_{2k}\hash M_{2k'}$ for $k'\leq k$. Recall that we have the identification
\begin{equation*}
\mathrm{Tor}^{\Rin}_{0,*}(M^*_{2k},M^*_{2k'})=M^*_{2k}\otimes_{\Rin}M^*_{2k'}.
\end{equation*}
The latter can be written as a direct sum of two $\Rin$-modules, one of which is
\begin{align*}
&\ztwo[V]\oplus\ztwo[V]\langle4k-1\rangle\oplus\ztwo[[V]]\langle4(k+k')-2\rangle\\
&\oplus\ztwo[V]/V^{k+k'}\langle4(k+k')-3\rangle\oplus\ztwo[V]/V^{k'}\langle4(k+k')-4\rangle\numberthis\label{mainsumm}
\end{align*}
where the action of $Q$ is not trivial from one summand to the one next to it on the right when there are two $\ztwo$ summands that differ in degree by one, and the other is $\ztwo[V]/V^{k'}\langle4k'-1\rangle$. The group $\mathrm{Tor}^{\Rin}_{1,*}$ is isomorphic to $\ztwo[V]/V^{k'}\langle4k'-4\rangle$. In the case where $k=2$ and $k'=1$, $\mathrm{Tor}^{\Rin}_{*,*}$ can be graphically described as follows:
\begin{center}
\begin{tikzpicture}
\matrix (m) [matrix of math nodes,row sep=0.5em,column sep=1em,minimum width=0.1em]
  {    \ztwo &\cdot&\cdot&\ztwo&\ztwo&\cdot&\cdot&\ztwo&\ztwo&\cdot&\ztwo_0&\ztwo&\ztwo&\cdots\\
  &\ztwo&\ztwo&&&\ztwo&&&&\ztwo\\
  &&&&&&&\ztwo\\
  &&&&&&&&&&\ztwo\\
};
\path[-stealth]

(m-1-1) edge[bend right] node {}(m-2-2)
(m-1-5) edge[bend right] node {}(m-2-6)
(m-1-9) edge[bend right] node {}(m-2-10)
;
\draw  (-6,-0.7) -- (6,-0.7);

\end{tikzpicture}
\end{center}
Here the first three rows represent $\mathrm{Tor}^{\Rin}_{0,*}$ (the first two rows being the summand in equation (\ref{mainsumm})) while the forth row represents $\mathrm{Tor}^{\Rin}_{1,*}$. For clarity, we have only depicted the $\Rin$-actions between elements of different rows. There are no trivial $\Rin$-extensions, so in particular we have
\begin{equation*}
\alpha(Y_k\hash Y_{k'})=2k+2k',\quad\beta(Y_k\hash Y_{k'})=2k,\quad\gamma(Y_k\hash Y_{k'})=0,
\end{equation*}
Notice that there are no non-trivial Massey products of the form $\langle V^k,\x,Q\rangle$ in the Floer homology of the connected sum. This implies that when taking a connected sum with $Y_{2k''}$, again only $\mathrm{Tor}^{\Rin}_{0,*}$ has to be taken account when computing correction terms. From here, a simple computation of the effect of tensoring with $M_{2k''}$ implies the claim.
\end{proof}

\vspace{0.3cm}
\textit{Connected sums with both orientations.} Consider the connected sum $M=M_1\hash M_1$, whose relevant part of the homology was computed in Example \ref{M1M1} to be
\begin{center}
\begin{tikzpicture}
\matrix (m) [matrix of math nodes,row sep=0.1em,column sep=0.7em,minimum width=0.1em]
  { \ztwo&\cdot &\ztwo&\ztwo&\ztwo&\cdots  \\
  &\ztwo\\};

\path[-stealth]

(m-1-1) edge[bend right] node {}(m-2-2)
(m-1-4) edge[bend right] node {}(m-1-5)
(m-1-3) edge[bend right] node {}(m-1-4)
(m-1-1) edge[bend left] node {}(m-1-5)

;
\end{tikzpicture}
\end{center}
Let us consider the dual $-(M_1\hash M_1)$, which has Floer homology depicted as
\begin{center}
\begin{tikzpicture}
\matrix (m) [matrix of math nodes,row sep=0.1em,column sep=0.7em,minimum width=0.1em]
  { \cdot&  \ztwo_{-1} &\ztwo  &\cdot&\ztwo& \ztwo &\ztwo&\cdots\\
  \ztwo_0\\};

\path[-stealth]

(m-1-2) edge[bend left] node {}(m-1-6)
(m-1-3) edge[bend left] node {}(m-1-7)
(m-1-2) edge[bend right] node {}(m-1-3)
(m-1-5) edge[bend right] node {}(m-1-6)
(m-1-6) edge[bend right] node {}(m-1-7)
(m-2-1) edge[bend right, dotted] node{}(m-1-3)
;
\end{tikzpicture}
\end{center}
Here the dotted line represents the triple Massey product $\langle \cdot,Q,Q^2\rangle$, which is determined by inspecting the Gysin exact sequence as in Proposition \ref{froy}. Consider now $M'=-(M_1\hash M_1)\hash M_2$. Using the description for connected sums with $M_2$ of Section \ref{connsimple}, we see that the relevant part of its Floer homology is given by
\begin{center}
\begin{tikzpicture}
\matrix (m) [matrix of math nodes,row sep=0.1em,column sep=0.7em,minimum width=0.1em]
  {&   \ztwo_1  &\cdot&\ztwo& \ztwo &\ztwo&\cdots\\
  \ztwo&&\ztwo_0\\};

\path[-stealth]

(m-1-2) edge[bend left] node {}(m-1-6)
(m-1-2) edge[bend right] node {}(m-2-3)
(m-1-5) edge[bend right] node {}(m-1-6)
(m-1-4) edge[bend right] node {}(m-1-5)
(m-2-1) edge[bend right,dotted] node {}(m-2-3)

;
\end{tikzpicture}
\end{center}
where the element in degree $2$ is the tensor product of the element of degree zero in the homology of $-(M_1\hash M_1)$ and $qv^{-1}$. In particular, it comes with a non trivial Massey product onto the generator in degree $0$, as depicted by the dotted arrow. In particular, $M$ and $M'$ can be distinguished up to homology cobordism just by looking at Massey products. Of course, in this case we already know that they cannot be homology cobordant because $M_1$ and $M_2$ are independent.
\\
\par
\textit{Imposing correction terms. }Let us discuss the following result (which should be compared with the analogous one in involutive Heegaard Floer homology from \cite{DaiS}).
\begin{prop}
Consider integers $a,b,c,d$ such that:
\begin{itemize}
\item $c\leq b,d\leq a$;
\item $a,b,c$ have the same parity.
\end{itemize}
Then there exists a homology sphere $Y$ with $\alpha(Y)=a$, $\beta(Y)=b$, $\gamma(Y)=c$ and $\delta(Y)=d$.
\end{prop}

Let us begin by discussing a slightly more involved example, namely $M_4\hash -M_3\hash-M_3$. To compute this we will regroup it as $M_4\hash -(M_3\hash M_3)$. The relevant part of $M_3\#M_3$ was described in Example \ref{M3M3}, so that by Poincar\'e duality the relevant part of $-(M_3\hash M_3)$ is
\begin{center}
\begin{tikzpicture}
\matrix (m) [matrix of math nodes,row sep=0.7em,column sep=1em,minimum width=0.1em]
  { &&\ztwo&\cdot&\cdot&\ztwo&\ztwo &\cdot &\cdot &\ztwo&\ztwo&\cdot &\ztwo&\ztwo&\ztwo&\cdots\\
  \ztwo_0&\cdot&\cdot&\cdot&\ztwo&\cdot&\cdot&\ztwo&\ztwo\\
};

\path[-stealth]

(m-1-3) edge[bend left] node {}(m-1-7)
(m-1-6) edge[bend left] node {}(m-1-10)
(m-1-7) edge[bend left] node {}(m-1-11)
(m-1-10) edge[bend left] node {}(m-1-14)
(m-1-11) edge[bend left] node {}(m-1-15)
(m-2-1) edge[bend left] node {}(m-2-5)
(m-2-5) edge[bend left] node {}(m-2-9)
(m-1-6) edge[bend right] node {}(m-1-7)
(m-1-10) edge[bend right] node {}(m-1-11)
(m-1-14) edge[bend right] node {}(m-1-15)
(m-2-8) edge[bend right] node {}(m-2-9)
(m-2-9) edge[dotted, bend right] node {}(m-1-13)
;
\end{tikzpicture}
\end{center}
Here the elements in the top row are based, while the elements in the bottom row are not. Denote by $\x$ the element in degree zero in the bottow row. We have highlighted with a dotted line the Massey product $\langle Q,V^2\x,V\rangle=v^3$, which will be needed later. When connecting sum with $M_4$, $\mathrm{Tor}_{1,*}^{\Rin}$ corresponds to the elements in $\mathrm{ker}V^2\cap \mathrm{ker}Q$, i.e. $V\x$ and $V^2\x$. In particular, they give rise to the two elements
\begin{equation}\label{tor1class}
V\x|V^2|qv^{-2}+V\x|Q|1,\qquad V^2\x|V^2|qv^{-2}+V^2\x|Q|1.
\end{equation}
The action of $V$ sends the first element to the second, and by Proposition \ref{Tor1} the second element is mapped via the action of $Q$ to $v^3|qv^{-2}\in \mathrm{Tor}_{0,*}^{\Rin}$. In particular, the computation of the module structure of the connected sum is readily obtained from that of $\mathrm{Tor}_{0,*}^{\Rin}$, i.e. the tensor product. The relevant part of the final result is
\begin{center}
\begin{tikzpicture}
\matrix (m) [matrix of math nodes,row sep=0.7em,column sep=1em,minimum width=0.1em]
  {&&&&&\ztwo&\cdot&\ztwo_0 &\cdot&\ztwo&\cdot &\underline{\ztwo}&\ztwo&\ztwo&\cdot&\underline{\ztwo}&\ztwo&\ztwo&\\
 \ztwo &\cdot&\ztwo&\cdot&\ztwo&\cdot&\ztwo\\
};

\path[-stealth]

(m-2-1) edge[bend left] node {}(m-2-5)
(m-2-3) edge[bend left] node {}(m-2-7)
(m-1-6) edge[bend left] node {}(m-1-10)
(m-1-10) edge[bend left] node {}(m-1-14)
(m-1-14) edge[bend left] node {}(m-1-18)
(m-1-6) edge[bend right] node {}(m-2-7)
(m-1-12) edge[bend right] node {}(m-1-13)
(m-1-13) edge[bend right] node {}(m-1-14)
(m-1-16) edge[bend right] node {}(m-1-17)
(m-1-17) edge[bend right] node {}(m-1-18)
;

\end{tikzpicture}
\end{center}
where the underlined $\ztwo$ summands correspond to the classes (\ref{tor1class}) from $\mathrm{Tor}_{1,*}^{\Rin}$. We have in this case
\begin{equation*}
\alpha=2,\quad \delta=0,\quad \beta=\gamma=-2.
\end{equation*} 
Again, the relevant Massey products can be inferred from both the tensor product formula or the Gysin exact triangle with
\begin{equation*}
\ztwo[U]\oplus(\ztwo[U]/U^4)_7\oplus (\ztwo[U]/U^3)_7\oplus (\ztwo[U]/U^3)_2
\end{equation*}
see also Example \ref{mccoy}.
In general, considering as in \cite{DaiS} connected sums of three manifolds with simple type with both orientations, one can construct homology spheres $Y$ with any given $\gamma,\leq \beta\leq\alpha$, $\alpha$, $\beta$ and $\gamma$ with the same parity and $\delta=0$. To conclude, one needs to take further connected sums with the Poincar\'e homology sphere $\Sigma(2,3,5)$: as $\HSf_{\bullet}(\Sigma(2,3,5))\cong\Rin\langle-3\rangle$ (see \cite{Lin2}), we have
\begin{equation*}
\HSf_{\bullet}(\Sigma(2,3,5)\hash Y)=\HSf_{\bullet}(Y)\langle-2\rangle,
\end{equation*}
so that all four correction terms are shifted down by $-1$.
\vspace{1cm}
\bibliographystyle{alpha}
\bibliography{biblio}

\begin{thebibliography}{KMOS07}

\bibitem[AS16]{AbS}
Mohammed Abouzaid and Ivan Smith.
\newblock The symplectic arc algebra is formal.
\newblock {\em Duke Math. J.}, 165(6):985--1060, 2016.

\bibitem[CGH12]{CGH1}
Vincent Colin, Paolo Ghiggini, and Ko~Honda.
\newblock The equivalence of {H}eegaard {F}loer homology and embedded contact
  homology via open book decompositions {I}.
\newblock {\em preprint}, arXiv:math/1208.1074, 2012.

\bibitem[Dai18]{Dai}
Irving Dai.
\newblock On the {${\rm Pin}(2)$}-equivariant monopole {F}loer homology of
  plumbed 3-manifolds.
\newblock {\em Michigan Math. J.}, 67(2):423--447, 2018.

\bibitem[DGMS75]{DGMS}
Pierre Deligne, Phillip Griffiths, John Morgan, and Dennis Sullivan.
\newblock Real homotopy theory of {K}\"ahler manifolds.
\newblock {\em Invent. Math.}, 29(3):245--274, 1975.

\bibitem[DM18]{DaiM}
Irving Dai and Ciprian Manolescu.
\newblock Involutive {H}eegaard {F}loer homology and plumbed three-manifolds.
\newblock {\em To appear in Journal of the Institute of Mathematics of
  Jussieu}, 2018.

\bibitem[DS17]{DaiS}
Irving Dai and Matthew Stoffregen.
\newblock On homology cobordism and local equivalence between plumbed
  manifolds.
\newblock {\em preprint}, 2017.

\bibitem[Eis80]{Eis}
David Eisenbud.
\newblock Homological algebra on a complete intersection, with an application
  to group representations.
\newblock {\em Trans. Amer. Math. Soc.}, 260(1):35--64, 1980.

\bibitem[Fur90]{FurCob}
Mikio Furuta.
\newblock Homology cobordism group of homology {$3$}-spheres.
\newblock {\em Invent. Math.}, 100(2):339--355, 1990.

\bibitem[Fy10]{Fro}
Kim~A. Fr\o~yshov.
\newblock Monopole {F}loer homology for rational homology 3-spheres.
\newblock {\em Duke Math. J.}, 155(3):519--576, 2010.

\bibitem[GM13]{GMor}
Phillip Griffiths and John Morgan.
\newblock {\em Rational homotopy theory and differential forms}, volume~16 of
  {\em Progress in Mathematics}.
\newblock Springer, New York, second edition, 2013.

\bibitem[GS99]{GS}
Robert~E. Gompf and Andr{\'a}s~I. Stipsicz.
\newblock {\em {$4$}-manifolds and {K}irby calculus}, volume~20 of {\em
  Graduate Studies in Mathematics}.
\newblock American Mathematical Society, Providence, RI, 1999.

\bibitem[HHL18]{HHL}
Kristen Hendricks, Jennifer Hom, and Tye Lidman.
\newblock Applications of involutive {H}eegaard {F}loer homology.
\newblock {\em preprint}, 2018.

\bibitem[HM17]{HenM}
Kristen Hendricks and Ciprian Manolescu.
\newblock Involutive {H}eegaard {F}loer homology.
\newblock {\em Duke Math. J.}, 166(7):1211--1299, 2017.

\bibitem[HMZ18]{HMZ}
Kristen Hendricks, Ciprian Manolescu, and Ian Zemke.
\newblock A connected sum formula for involutive {H}eegaard {F}loer homology.
\newblock {\em Selecta Math. (N.S.)}, 24(2):1183--1245, 2018.

\bibitem[KLT11]{HFHM1}
Cagatay Kutluhan, Yi-Jen Lee, and Clifford Taubes.
\newblock {HF}={HM} {I} : {H}eegaard {F}loer homology and {S}eiberg--{W}itten
  {F}loer homology.
\newblock {\em preprint}, arXiv:math/1007.1979, 2011.

\bibitem[KM07]{KM}
Peter Kronheimer and Tomasz Mrowka.
\newblock {\em Monopoles and three-manifolds}, volume~10 of {\em New
  Mathematical Monographs}.
\newblock Cambridge University Press, Cambridge, 2007.

\bibitem[KMOS07]{KMOS}
P.~Kronheimer, T.~Mrowka, P.~Ozsv{\'a}th, and Z.~Szab{\'o}.
\newblock Monopoles and lens space surgeries.
\newblock {\em Ann. of Math. (2)}, 165(2):457--546, 2007.

\bibitem[Ks80]{Kad}
T.~V. Kadei\v~svili.
\newblock On the theory of homology of fiber spaces.
\newblock {\em Uspekhi Mat. Nauk}, 35(3(213)):183--188, 1980.
\newblock International Topology Conference (Moscow State Univ., Moscow, 1979).

\bibitem[Lin15a]{Lin}
Francesco Lin.
\newblock {A} {M}orse-{B}ott approach monopole {F}loer homology and the
  {T}riangulation {C}onjecture.
\newblock {\em preprint}, arXiv:math/1404.4561, 2015.

\bibitem[Lin15b]{Lin2}
Francesco Lin.
\newblock {T}he surgery exact triangle in {P}in(2)-monopole {F}loer homology.
\newblock {\em preprint}, arXiv:math/1504.01993, 2015.

\bibitem[Lin16]{Lin3}
Francesco Lin.
\newblock {L}ectures on monopole {F}loer homology.
\newblock {\em To appear in the Proceedings of the G\"okova Geometry and
  Topology conference}, 2016.

\bibitem[Lin17a]{Lin6}
Francesco Lin.
\newblock Manolescu correction terms and knots in the three-sphere.
\newblock {\em https://arxiv.org/abs/1607.05220}, 2017.

\bibitem[Lin17b]{Lin4}
Francesco Lin.
\newblock {${\rm Pin}(2)$}-monopole {F}loer homology, higher compositions and
  connected sums.
\newblock {\em J. Topol.}, 10(4):921--969, 2017.

\bibitem[Lin18]{Lin5}
Francesco Lin.
\newblock {P}in(2)-monopole {F}loer homology and the {R}okhlin invariant.
\newblock {\em To appear in Compositio Mathematica}, 2018.

\bibitem[Lip14]{Lip}
Max Lipyanskiy.
\newblock Geometric homology.
\newblock {\em Preprint}, arXiv:math/1303.2354, 2014.

\bibitem[LP11]{LP}
Yank\i Lekili and Timothy Perutz.
\newblock Fukaya categories of the torus and {D}ehn surgery.
\newblock {\em Proc. Natl. Acad. Sci. USA}, 108(20):8106--8113, 2011.

\bibitem[Man13a]{Man3}
Ciprian Manolescu.
\newblock The {C}onley index, gauge theory, and triangulations.
\newblock {\em J. Fixed Point Theory Appl.}, 13(2):431--457, 2013.

\bibitem[Man13b]{Man2}
Ciprian Manolescu.
\newblock {P}in(2)-equivariant {S}eiberg-{W}itten {F}loer homology and the
  {T}riangulation {C}onjecture.
\newblock {\em preprint}, arXiv:math/1303.2354, 2013.

\bibitem[McC01]{McC}
John McCleary.
\newblock {\em A user's guide to spectral sequences}, volume~58 of {\em
  Cambridge Studies in Advanced Mathematics}.
\newblock Cambridge University Press, Cambridge, second edition, 2001.

\bibitem[MOY97]{MOY}
Tomasz Mrowka, Peter Ozsv{\'a}th, and Baozhen Yu.
\newblock Seiberg-{W}itten monopoles on {S}eifert fibered spaces.
\newblock {\em Comm. Anal. Geom.}, 5(4):685--791, 1997.

\bibitem[OS04]{OS2}
Peter Ozsv{\'a}th and Zolt{\'a}n Szab{\'o}.
\newblock Holomorphic disks and three-manifold invariants: properties and
  applications.
\newblock {\em Ann. of Math. (2)}, 159(3):1159--1245, 2004.

\bibitem[Ser00]{Ser}
Jean-Pierre Serre.
\newblock {\em Local algebra}.
\newblock Springer Monographs in Mathematics. Springer-Verlag, Berlin, 2000.
\newblock Translated from the French by CheeWhye Chin and revised by the
  author.

\bibitem[Sto15a]{Sto2}
Matthew Stoffregen.
\newblock Manolescu invariants of connected sums.
\newblock {\em preprint}, 2015.

\bibitem[Sto15b]{Sto}
Matthew Stoffregen.
\newblock {P}in(2)-equivariant {S}eiberg-{W}itten {F}loer homology of {S}eifert
  fibrations.
\newblock {\em preprint}, 2015.

\bibitem[Sto17]{Sto3}
Matthew Stoffregen.
\newblock A remark on {${\rm Pin}(2)$}-equivariant {F}loer homology.
\newblock {\em Michigan Math. J.}, 66(4):867--884, 2017.

\bibitem[Val12]{Val}
Bruno Vallette.
\newblock Algebra + {Homotopy}={O}perad.
\newblock {\em Preprint}, arXiv:math/1202.3245, 2012.

\end{thebibliography}

\end{document}